\documentclass[10pt,draft]{siamltex1213}
\usepackage{amsmath, amssymb,enumerate}
\newcommand{\new}[1]{#1}

\newtheorem{assumption}[theorem]{Assumption}

\DeclareMathOperator\proj{proj}

\newcommand{\oo}[1]{\tfrac{1}{#1}}
\newcommand{\abs}[1]{\left| #1 \right|} 
\newcommand{\set}[1]{\{#1\}} 
\newcommand{\sets}[2]{\set{#1\,:\,#2}} 
\newcommand{\bset}[1]{\big\{#1\big\}} 
\newcommand{\bsets}[2]{\bset{#1\,:\,#2}} 
\newcommand{\Bset}[1]{\Big\{#1\Big\}} 
\newcommand{\Bsets}[2]{\Bset{#1\,:\,#2}} 
\newcommand{\ind}[1]{ {\mathbf 1}_{{#1}}} 
\newcommand{\inds}[1]{ {\mathbf 1}_{\set{#1}}} 
\newcommand{\seq}[1]{\set{#1_n}_{n\in\N}} 
\renewcommand{\fam}[2]{\{ #1 \}_{#2}}

\newcommand{\prf}[1]{ \{ #1 \}_{t\in [0,T]}}

\newcommand{\prfi}[1]{ \{ #1 \}_{t\in [0,\infty)}}

\newcommand{\dd}{d}

\newcommand{\tRN}[2]{\tfrac{\dd #1}{\dd #2}}

\providecommand{\R}{} \renewcommand{\R}{{\mathbb R}}
 
\newcommand{\N}{{\mathbb N}}
\newcommand{\PP}{{\mathbb P}}
\newcommand{\QQ}{{\mathbb Q}}
\newcommand{\hQ}{\hat{\QQ}}
\newcommand{\EE}{{\mathbb E}}
\newcommand{\FF}{{\mathcal F}}

\newcommand{\eeq}[2]{ \bE^{#1}\left[ #2 \right] }

\newcommand{\Beeq}[2]{ \bE^{#1}\Big[ #2 \Big] }
\newcommand{\eecq}[3]{ \bE^{#1}\left[ #2 | #3 \right] }

\newcommand{\MM}{{\mathcal M}}

\newcommand{\FFF}{{\mathbb F}}

\renewcommand{\AA}{{\mathcal A}}

\newcommand{\eps}{\varepsilon}
\newcommand{\ld}{\lambda}

\newcommand{\gm}{\gamma}
\newcommand{\vp}{\varphi}
\newcommand{\esl}{{\mathcal L}} 

\newcommand{\slzerone}{\esl^{0-1}}

\newcommand{\el}{{\mathbb L}} 

\newcommand{\lone}{\el^1}

\newcommand{\define}[1]{{\textbf{#1}}}

\newcommand{\efor}{\text{ for }}
\newcommand{\eand}{\text{ and }}
\newcommand{\ewhere}{\text{ where }}

\newcommand\tx{{\tilde{x}}}

\newcommand\sA{{\mathcal A}}

\newcommand\sB{{\mathcal B}}

\newcommand\tB{{\tilde{B}}}

\newcommand\bE{{\mathbb E}}
\newcommand\tE{{\tilde{E}}}

\newcommand\sF{{\mathcal F}}

\newcommand\bF{{\mathbb F}}

\newcommand\sG{{\mathcal G}}

\newcommand\sK{{\mathcal K}}

\newcommand\sM{{\mathcal M}}

\newcommand\sN{{\mathcal N}}

\newcommand\sP{{\mathcal P}}

\newcommand\sQ{{\mathcal Q}}

\newcommand\sR{{\mathcal R}}

\newcommand\sS{{\mathcal S}}

\newcommand\sT{{\mathcal T}}

\newcommand\sU{{\mathcal U}}

\newcommand\tU{{\tilde{U}}}

\newcommand\tX{{\tilde{X}}}

\newcommand\hZ{{\hat{Z}}}
\newcommand\AsE{\mathrm{AE}}
\newcommand{\tsP}{\tilde{\sP}}
\newcommand{\zpq}{Z^{\PP,\QQ}}
\newcommand{\hPP}{\hat{\PP}}
\newcommand{\ttau}{\tilde{\tau}}
\newcommand{\lht}{ L\times  (H\circ\theta_{\ttau})}
\newcommand{\tFFF}{\tilde{\FFF}}

\newcommand{\Xd}{X_{\cdot}}
\newcommand{\Zd}{Z_{\cdot}}
\newcommand{\prob}{\mathfrak{P}}
\newcommand{\tprob}{\tilde{\mathfrak{P}}}

\newcommand{\fR}{{\mathfrak R}}

\newcommand{\cct}{*_{\tau}}
\newcommand{\cctt}{*_{\tilde{\tau}}}

\newcommand{\ott}{\otimes_{\tau}}

\newcommand{\nll}{\not\ll}
\newcommand{\hQQ}{\hat{\QQ}}
\newcommand{\hmu}{\hat{\mu}}
\newcommand{\hne}{\hat{\nu}^{\eps}}

\newcommand{\de}{D_{E}}
\newcommand{\dz}{D_{[0,\infty)}}

\newcommand{\dep}{D_{\tE}}

\newcommand{\bde}{\sB(\de)}
\newcommand{\spx}{\sP^x}

\newcommand{\spxe}{(\spx)_{x\in E}}
\newcommand{\smx}{\sM^x}
\newcommand{\smxe}{(\smx)_{x\in E}}
\newcommand{\selp}{\sS(\sP)}
\newcommand{\gp}{\Gamma_{\sP}}
\newcommand{\tgp}{\Gamma_{\tsP}}
\newcommand{\gpe}{\Gamma_{\sP_{\eps}}}
\renewcommand{\gm}{\Gamma_{\sM}}
\newcommand{\gmn}{\Gamma_{\sM_n}}

\newcommand{\qcn}{\QQ\cct\nu}

\newcommand{\hnu}{\hat{\nu}}

\title{Dynamic Programming for controlled Markov families: abstractly and
over Martingale Measures\thanks{The author would
like to thank Gerard Brunick, Kasper Larsen and Mihai S\^\i rbu for
numerous and helpful discussions on the subject of dynamic
programming.  Careful reading by the anonymous referees and the associate
editor of the original manuscript, as well as their numerous
suggestions for improvement, is gratefully acknowledged.
This material is based upon work supported by the National
Science Foundation under Grants No.  DMS-0706947 (2010 - 2015) and Grant
No.~DMS-1107465 (2012 - 2017).  Any opinions, findings and conclusions or
recommendations expressed in this material are those of the author(s) and
do not necessarily reflect the views of the National Science Foundation
  (NSF)}} 

\author{Gordan {\v Z}itkovi{\' c}\footnotemark[2]}

\begin{document}
\maketitle
\renewcommand{\thefootnote}{\fnsymbol{footnote}}
\footnotetext[2]{Department of Mathematics, The University of Texas at
Austin} 
\renewcommand{\thefootnote}{\arabic{footnote}}

\begin{abstract} 
  We describe an abstract control-theoretic framework in which the validity
  of the dynamic programming principle can be established in continuous
  time by a verification of a small number of structural properties. As an
  application we treat several cases of interest, most notably the
  lower-hedging and utility-maximization problems of financial mathematics
  both of  which are naturally posed over ``sets of martingale measures''.
\end{abstract} 


\pagestyle{myheadings}
\thispagestyle{plain}
\markboth{G.~{\v Z}ITKOVI{\' C}}{DYNAMIC PROGRAMMING FOR CONTROLLED MARKOV FAMILIES}

\section{Introduction} The goal of this article is to prove the dynamic
programming principle (DPP in the sequel) for a class of stochastic control
problems in a continuous-time Markov-like environment, including two
fundamental problems in financial mathematics defined over or parametrized
by a set of ``martingale'' measures.  To this end, we introduce a flexible
abstract framework in which, we hope, other problems can be treated as
well. 

The history of the DPP is fascinating; it starts with the work of Wald
\cite{Wal50} and Bellman \cite{Bel57a}, although the idea of the reduction
of a complicated sequential problem to a family of simpler ones is
undoubtedly much older. A rigorous discrete-time theory focusing on
intricate measurability issues started with the work of Blackwell and
others (see, e.g., \cite{Bla62, Bla65, Str66, Bla67, BlaFreOrk74,
BerShr79}; we refer the reader to \cite{BerShr96} for further
bibliographical information and a detailed discussion).  It was already
known to these authors that -  unlike in the setting of deterministic
optimal control, where the validity of the DPP is easier to establish - the
stochastic version of the DPP comes with a number of additional subtleties
and requires a much more delicate treatment.  For this reason, two main
schools of thought dominate the control-theoretic literature.  In one of
them, great importance is given to a rigorous derivation of an appropriate
version of the DPP.  The other, however, treats the very need for a proof
of the DPP as a mathematical pedantry and sees it as an evidently correct
general principle. 

In order to provide a suitable general theoretical foundation \emph{in
continuous time}, and, in the endgame, a practical convergence of the two
approaches, a flurry of activity over the last decades expanded
tremendously our understanding of continuous-time DPP in a variety of
settings (see, e.g., \cite{Elk81}, \cite{Bor89}, \cite{FleSon93},
\cite{SonTou02}, \cite{SonTou02a}, \cite{BouVu10}, \cite{BouTou11},
\cite{BouNut12}, etc.). \new{ Independently of the present paper, and with
a different point of view, the authors of \cite{ElKTan13} and
\cite{ElKTan13a} also provide an abstract approach to the dynamic
programming and give a thorough treatment of descriptive set-theoretic
machinery which underlies it.  Moreover, a similar
descriptive-set-theoretic idea have recently been used successfully in the
context of model uncertainty and nonlinear expectations (see \cite{Nut12},
\cite{NutHan13} and \cite{BouNut13}), with the setting and technique,
especially in \cite{NutHan13},  similar to that of parts of Sections 2.1
and 3.2 of the present paper.} 

A different point of view, called the ``Stochastic Perron Method'',
initiated by Bayraktar and Sirbu \new{(see \cite{BaySir12},
\cite{BaySir12a} and \cite{BaySir12b})} circumvents the use of the formal
DPP altogether and derives verification results for the corresponding HJB
equations directly under natural condition on the equation itself.

\subsection{Weakly-constrained problems and the convenience of the weak
formulation} While quite general and extremely useful in their own domains,
none of the existing results seem to apply to the general Markovian
versions of the problems of lower hedging or utility maximization, often
seen as fundamental in the field of financial mathematics. These problems
are naturally defined over sets of equivalent local-martingale measures
and, depending on the particular setting, do not admit a naive translation
into the classical control milieu.  Indeed, they usually come in the
``weakly-constrained'' form: when framed in the classical
stochastic-control language, the set of admissible controls is
unconstrained locally, but still required to consist of processes which,
when acted upon by a stochastic-exponential type operator, yield
uniformly-integrable (as opposed to merely local) martingales.  This
difficulty vanishes when one chooses to view these problems in their (we
dare say, even more natural) weak formulation, and works with sets of
probability measures rather than control processes right from the start.
The weak approach is certainly not new in stochastic control theory; the
first uses of a ``controlled martingale problem'' go back at least to
\cite{Bis76}.  Many other authors (see \cite{YonZho99} and \cite{FleSon06}
and the references therein) use similar concepts under different names,
such as the "weak formulation".

An added technical benefit of the weak approach is the complete avoidance
of subtle measure-theoretic difficulties (mostly dealing with the
stochastic integration) inherent to the DPP-based treatment of
strongly formulated problems. In fact, it could be argued that the entire
``philosophy'' of our approach is based on this fact: the
descriptive-set-theoretic and topological framework of classical
discrete-time optimal control and the analytical and process-theoretic
framework of the contemporary stochastics do not seem to play well
together. For example, the filtration completion makes a tractable
stochastic-integration theory possible, while simultaneously  destroying
the countable-generation property. 

The strong formulation - where the controls and the noise inhabit a
fixed filtered 
probability space - may seem to be more amenable to concatenation (the
basic construction in any DPP treatment). The message we are sending here
is that, from the analytical point of view, measures are easier to work
with, and just as easy to concatenate. This perspective aligns well with
the mental model often present in the classical gambling theory (see
\cite{MaiSud96}), wherein the player (controller) chooses a ``gambling
house``, i.e., one among possibly many available probability measures to
govern the future evolution of the state process. Clearly, the difference
lies mostly in interpretation, but, as we hope our results demonstrate, it
can be an important one. 

\subsection{An abstract version of the DPP} To implement the conceptual
framework described above, we introduce a structure, called the
\define{controlled Markov family}, reminiscent of, and modeled on,  a
classical Markov family associated to a Markov process.  It consists of a
family of sets of  probability measures, one for each  point of the state
space $E$, with the interpretation that the probability measures available
at $x\in E$ can be used by the controller in order to steer the system so
as to minimize the expected value of the given cost functional. These
probability measures are defined on the set $\de$ of RCLL trajectories and
do not necessarily form a Markov family, or correspond to a Markov process
on $E$. In fact, our main theorem can be understood as a description of
how close to ``being Markov'' our families of sets of probability
measures (one for each $x\in E$) must be in order for the DPP to hold.
Put yet differently, we are interested in a useful generalization of
Chapman-Kolmogorov equations to the set-valued case.  

It turns out that two natural requirements, namely \define{concatenability}
(closure under concatenation) and \define{analyticity}, lead to a
convenient definition.  The first makes sure that the controller can change
his/her mind midstream and switch to any measure available at the current
state.  The second condition is of purely technical nature, and imposes a
minimal degree of regularity on how the family of available probability
laws changes from point to point. To understand it better, let us mention
that analyticity directly implies that the value function is upper
semi-analytic, which, in turn, makes it universally measurable.  As
observed in \cite{ShrBer79}, universal measurability (i.e., measurability
with respect to a completion of the Borel $\sigma$-algebra under any Borel
probability measure) is about the weakest property one can require from the
value function for even the formulation of the DPP to make sense. Indeed,
the right-hand side of a typical DPP statement will involve the expectation
(an integral) of the value function applied to a random variable, which,
without a minimal degree of measurability, cannot even be defined without
running into serious  conceptual difficulties.

While the value functions associated with controlled Markov processes
already possess many good properties and satisfy ``half'' of the DPP,
another condition -  in a sense dual to concatenability -  is needed for
the full DPP to hold. Named \define{(approximate) disintegrability}, it
postulates that, if required to switch to an admissible control, the
controller can oblige at any point in time and still remain (nearly)
optimal. Our main abstract result states that a controlled Markov family
with the additional property of disintegrability defines a value function
for which the full DPP holds and extends the classical
descriptive-set-theoretic approach of discrete-time stochastic optimal
control to continuous time with Markovian-like dynamics.  As we explain
it in more detail in subsection \ref{sub:fex}, part (1), that canonical
Feller families (and virtually all other Markov families under minimal
regularity conditions), when understood as controlled Markov families
with singleton control sets, automatically satisfy the three key
properties of analyticity, concatenability and disintegrability. 

\new{It is important to make it clear that our main abstract result -
Theorem \ref{thm:main} -  together with the techniques used in its proof,
is quite similar to some of the related results found throughout the
literature (some of them going back to Blackwell). On the other hand, its
scope and the level of abstraction are new and seem to be unavailable; its
value lies primarily in its applicability in specific situations, as
explained below.}

\subsection{Examples and applications to financial mathematics} After the
main abstract theorem and a practical sufficient condition for its
assumptions to hold, the paper treats several examples in various degrees
of detail. We start by showing that many (uncontrolled) Markov families -
such as the canonical RCLL Markov families - naturally satisfy the
conditions of concatenability, disintegrability and analyticity, in a
rather trivial way. We comment, though, that the families obtained by
solving martingale problems with multiple solutions furnish a less trivial
example: the family of sets of all weak solutions of a martingale problem
will form a controlled Markov family under the right regularity conditions,
too.  Next, we turn to a bird's-eye discussion of how one can interpret
control problems in the classical strong formulation as controlled Markov
families, and what structure is needed for our abstract theorem to apply. 

The technical bulk of the paper is devoted to a detailed treatment of the
lower-hedging and the utility-maximization problems of financial
mathematics.  It is in these - but certainly not exclusively in these -
moderately nonstandard control problems where the power of our approach
seems to be evident.  The first example we focus on deals with the problem
of lower hedging (sub-replication) in rather general financial markets. We
showcase the power of our framework by establishing a DPP for this problem
when the stock prices, along with  non-traded factor processes, form a RCLL
Feller process with a Hausdorff LCCB state space. Additional conditions
imposed on the model are minimal: the classical no-arbitrage (NFLVR)
assumption and the weak requirement of local boundedness from below (which
is, e.g., implied by continuity or nonnegativity of the stock price).
Under these conditions, the collections of equivalent local-martingale
measures - parametrized by the initial condition - form a controlled Markov
family and, additionally, satisfy the condition of disintegrability. 

The results obtained for lower hedging are further developed to show that
the DPP holds for the utility-maximization problem with a random endowment
under the same, minimal, regularity assumptions. The way this problem is
tackled sheds additional light on how our framework may be used: unlike in
the lower hedging problem, the set of equivalent local martingale measures
itself \emph{does not} play the role of the controlled Markov family, but
it \emph{parametrizes it} in a Borel-measurable way. Thanks to good
stability properties of analytic sets, this is enough to guarantee the
assumptions of the abstract theorem and, consequently, imply the DPP for
this problem, too. 
\subsection{The structure of the paper} After this introduction, section 2
outlines the abstract setup and presents our main abstract theorem together
with its proof. First examples and a necessary condition for the
assumptions of the theory to hold are also given.  Section 3 applies the
abstract results to the problems of lower hedging and utility maximization
in financial mathematics. A short appendix contains a telegraphic primer
(in place primarily to fix the terminology) on the pertinent concepts and
results of basic descriptive set theory. 

\bigskip

\section{An abstract framework for the dynamic programming principle}\ 
\label{sec:model}
We start by describing a fairly general probabilistic setting for the
weak formulation of optimal stochastic control.  The reader will
observe that - in its essence, at least -  it is a marriage between
the classical discrete-time framework as described, for example, in
\cite{BerShr78}, and the standard set-up for Markov families on the
canonical space of right-continuous paths with left limits, as
outlined, for instance,  in \cite{EthKur86}.  In fact, an effort has
been made to stress the formal analogy with the abstract structure of
a Markov family as much as possible. We remind the reader that a short
glossary of necessary descriptive-set-theoretic terms (used below) is
given in the appendix.

\subsection{Elements of the setup} For completeness and definiteness, we
start by defining some standard elements of the setup. 

\medskip

\paragraph{The space $\de$} Given a Polish space $E$ (a topological
space homeomorphic to a complete separable metric space), let $\de$ be the
set of all  $E$-valued RCLL paths on $[0,\infty)$.  A generic path in $\de$
is typically denoted by $\omega$ and the map $X_t:\de\to E$, defined for
$t\in [0,\infty)$ and $\omega\in\de$ by $X_t(\omega)=\omega(t)$, is
referred to as the \define{coordinate mapping}.  It is always assumed that
an isolated point $\iota$ is adjoined to $E$ (in addition to any
``cemetery''-type points already present to deal with ``killing'' and
sub-probability measures).  This way, the index set for the coordinate
mappings can be naturally extended to the class of all maps $\tau:\de\to
[0,\infty]$ by setting $X_{\tau}(\omega)=X_{\tau(\omega)}(\omega)$, when
$\tau(\omega)<\infty$; and $X_{\tau}(\omega)=\iota$, otherwise.

The family $(\theta_t)_{t\geq 0}$ of \define{shift operators} is defined on
$\de$ by $X_s(\theta_{t}(\omega)) = X_{t+s}(\omega)$, for all $s,t\geq 0$
and $\omega\in\de$.  Just like in the case of the coordinate mappings,
definitions of  shift operators can be extended to include
random-time-valued indices, namely by setting $X_t(\theta_{\tau}(\omega))
=X_{\tau(\omega)+t}(\omega)$; here $\theta_{\tau}(\omega)=\omega_{\iota}$
when $\tau(\omega)=+\infty$ with $\omega_{\iota}$ denoting the constant
trajectory with the value $\iota$.  Informally, unless otherwise stated, we
imagine all paths as taking the constant value $\iota$ ``at'' and ``after''
$+\infty$. 

The space $\de$ is always assumed to be equipped with the Skorokhod
topology; this way it inherits the structure of a Polish space itself (see
\cite[Theorem 5.6, p.~121]{EthKur86}).  While quite important for the
Polish property of $\de$, the choice of the Skorokhod topology is not
crucial as far as the Borel $\sigma$-algebra $\bde$, generated by it on
$\de$, is concerned.  Indeed, it can be shown that it is generated by most
other, often used, metrics on $\de$; in fact, $\bde$ is simply the
$\sigma$-algebra generated by the coordinate maps (see \cite[Proposition
7.1, p.~127]{EthKur86}).

\medskip

\paragraph{Probability measures on $\de$} The set of all probability
measures on $\de$ is denoted by $\prob(\de)$ or, simply, by $\prob$, if no
confusion is anticipated. Since we only consider the Borel $\sigma$-algebra
$\sB(\de)$ on $\de$, we permit ourselves to abuse the language in the usual
way, and refer to $\mu\in\prob$ as a ``probability on $\de$'', as opposed
to ``probability on $\sB(\de)$''. As usual, $\prob$ is endowed with the
topology of weak convergence, which gives it the structure of a Polish
space (see \cite[Theorem 1.7, p.~101]{EthKur86}) and generates the Borel
$\sigma$-algebra $\sB(\prob)$.  

We simplify the exposition by adopting some of the probabilistic
terminology and notation. For example, depending on the context, we use
either the probabilistic $\EE^{\mu}[G]$ or the analytic $\int G\, d\mu$ (or
$\int G(\omega)\, \mu(d\omega)$) notation to denote the integral of the
appropriately measurable function $G$  with respect to the probability
measure $\mu$ over $\de$.  In general, the measurable space $(\de,\bde)$,
together with some probability measure $\mu\in\prob$, will play the role of
the underlying probability space for the remainder of the paper; when
probabilistic terminology is used, it will be with respect to this space.

\medskip

\paragraph{Universally measurable kernels and random variables}
\label{sss:kernels} A map $\nu:E\times \bde\to [0,1]$ is called a
\define{universally measurable kernel} (or, simply, \define{kernel}) if
\begin{enumerate}
  \item	$\nu(x,\cdot)\in\prob$ for all $x\in E$,  and
  \item $E\ni x\mapsto \nu(x,A)$ is universally 
	measurable for all $A\in\bde$. 
\end{enumerate}	
Due to the special nature of the state $\iota$, we always assume that
$\nu_{\iota}(A)= \inds{\omega_{\iota}\in A}$, i.e., that $\nu_{\iota}$ is
necessarily the Dirac mass at the constant trajectory $\omega_{\iota}$.

To make the notation easier to read, we often write $\nu_x$ for the
probability measure $\nu(x,\cdot)$ and interpret $\nu$ as a map $E\to
\prob$.  A classical result of Varadarajan (see \cite[Lemma 2.3.,
p.~194]{Var63}) states that a universally measurable kernel, when
interpreted this way, defines a universally-measurable map $E \to \prob$.
Since $\prob$ is Polish, the \define{graph} $\Gamma_{\nu}$ of $\nu$, given
by $\Gamma_{\nu}=\sets{ (x,\nu_x)}{ x\in E}\subseteq E\times \prob$ is a
product-measurable subset of $E\times \prob$, when $E$ is equipped with the
universally-measurable $\sigma$-algebra. 

In the spirit of our interpretation of $(\de,\bde)$ as a probability space,
a universally-measurable map $G:\de\to [-\infty,\infty]$ is called a
\define{universally measurable random variable} (or, simply,
\define{random variable}).  Unless otherwise specified, we always set
$G(\iota)=0$.  

Given a family $\mathfrak{R}\subset \prob$ of probability measures on
$\de$, a random variable $G$ is said to be \define{$\mathfrak R$-lower
semi-integrable} if $\EE^{\mu}[G^-]<\infty$, for all $\mu\in\mathfrak{R}$;
the set of all $\mathfrak{R}$-lower semi-integrable random variables is
denoted by $\slzerone(\mathfrak{R})$.  For $\mu\in\prob$ and $G\in
\slzerone(\set{\mu})$, the expectation $\EE^{\mu}[G]$ is well-defined, with
values in $(-\infty,\infty]$ by setting $\EE^{\mu}[G] =
\EE^{\mu}[G^+]-\EE^{\mu}[G^-]$.  We note that, given a kernel $\nu$ and a
random variable $G\in\slzerone( \fam{\nu_x}{x\in E})$, the integral
$g(x)=\int G(\omega) \nu_x(d\omega)$ defines a universally measurable map
$g: E\to (-\infty,\infty]$. 

\medskip

\paragraph{Concatenation of paths} A Borel-measurable map from
$\de$ to $[0, \infty]$ is called a \define{random time}. For a random
time $\tau$ and two paths $\omega,\omega'\in\de$, the
\define{concatenation $\omega \cct \omega'$ of $\omega$ and $\omega'$
at $\tau$} is an element of $\de$ whose value at $t\geq 0$ is given by 
\[ X_t(\omega \cct \omega') = 
  \begin{cases} 
	X_t(\omega), & t< \tau(\omega), \\ 
	X_{\tau}(\omega)+X_{t- \tau(\omega)}(\omega') -X_0(\omega'), &
	t\geq \tau(\omega).
  \end{cases} \] 
It follows immediately that for $\omega,\omega'\in\de$ and a random time
$\tau$, we have 
\begin{equation}
  \label{equ:compose}
  \begin{split}
	\omega\cct \theta_{\tau}(\omega) =
	\omega, \text{ and }  \theta_{\tau} ( \omega\cct \omega' ) 
	= \begin{cases}
	  \new{\omega' + \Big(X_\tau(\omega) - X_0(\omega')\Big)}, & \tau(\omega)<\infty\\ \omega_{\iota}, &
	  \tau(\omega)=\infty.  \end{cases}
	\end{split}
  \end{equation} 
The map $(\omega,\omega')\mapsto X_t(\omega \cct \omega')$ is easily seen
to be $(\bde\otimes\bde,\sB(E))$-measurable for each $t\geq 0$.  Thanks to
the fact that the Borel $\sigma$-algebra $\bde$ on $\de$ is generated by
the family of coordinate maps,  the concatenation map $\de\times \de\to
\de$ is Borel-measurable, as well. 

\medskip

\paragraph{Concatenation of laws} For a random time $\tau$, a
probability measure $\mu\in\prob$ and a kernel $\nu$ we define the
\define{product} $\mu\ott \nu$ as the probability measure on $\de\times\de$
with
\begin{equation}
  \label{equ:ott}
  \begin{split}
	(\mu\ott\nu)[B] = \iint
  \ind{B}(\omega,\omega')\,
  \nu_{X_{\tau}(\omega)}(d\omega')\, \mu(d\omega),\text{ for }
B\in \sB(\de\times\de).
 \end{split}
\end{equation}
The \define{concatenation $\mu\cct \nu$} of $\mu$ and $\nu$ at $\tau$ is
the probability measure $\mu\cct\nu$ on $\de$, given by 
\begin{equation}
  \label{equ:cct}
  \begin{split}
	(\mu\cct \nu)[A] = \iint \ind{A}(\omega\cct\omega')\,
	\mu\ott\nu (d\omega,d\omega'),\text{ for } A\in\bde.
 \end{split}
\end{equation}
We observe that the fact that $X_{\tau}=\iota$ on $\set{\tau=\infty}$ plays
a minimal role in the definition \eqref{equ:cct} of $\mu\cct\nu$ above.
Indeed, if $\tau(\omega)=\infty$, we have $\omega\cct\omega'=\omega$, and
the inner integral is applied to the constant function $\omega'\mapsto
\ind{A}(\omega)$. Therefore, its value is simply $\ind{A}(\omega)$ as soon
as $\nu_{X_{\tau}}$ is (an arbitrary)  probability measure.  Accordingly,
for a sufficiently integrable (nonnegative, for example) random variable
  $G$,  we have
\begin{equation}
  \label{equ:cct-G}
  \begin{split}
	\EE^{\mu\cct\nu}[G] = \EE^{\mu}[ \tilde{G}],\text{ where }
  \tilde{G}(\omega) = \begin{cases}
	\EE^{\nu_{X_{\tau}(\omega)}}[ G(\omega\cct\cdot)], &
        \tau(\omega)<\infty, \\ G(\omega), & \tau(\omega)=\infty.\end{cases}
 \end{split}
\end{equation}
In particular, we note that for $G\in\slzerone(\sets{\mu\otimes
\nu}{\mu\in\fR })$, we have $\tilde{G}\in\slzerone(\fR)$, for any
collection $\fR\subseteq \prob$.

\medskip

\paragraph{Controlled Markov families}
Let $\sP=\spxe$ be a family
of  non-empty subsets of $\prob$.  A universally measurable kernel
$\nu$ with $\nu_x\in \spx$, for all $x\in E$, is called a
\define{$\sP$-selector}; the set of all $\sP$-selectors is 
denoted by $\selp$. We note that a kernel $\nu$ is a $\sP$-selector
if and only if  $\Gamma_{\nu} \subseteq \Gamma_{\sP}$, where
$\gp = 
\sets{(x,\mu)}{ \mu\in\sP^x} \subseteq E\times \prob$. 

\medskip

\begin{definition}
  \label{def:conMarfam}
  A pair $(\sP,\sT)$, consisting of a family 
  $\sP=\spxe$ of nonempty subsets of $\prob$ and a set $\sT$ of random
  times is called  a \define{controlled Markov family} if 
  \begin{enumerate}
	\item \label{ite:a-meas} 
	  the graph $\gp=\bsets{ (x,\mu)}{ \mu\in \spx}$ is an analytic
	  subset of $E\times \prob$, and 
	\item \label{ite:concat} 
	  for $x\in E$, $\mu\in\spx$, $\tau\in\sT$ and $\nu\in\sS(\sP)$, we
	  have $\mu\cct \nu \in \spx$. 
\end{enumerate}
We refer to the first property above as \textbf{analyticity} and to the
second one as \textbf{concatenability}. 
\end{definition}

\medskip

We think of a controlled Markov family as a stripped-down version
 of a mechanism by which a stochastic system is controlled; a measure
 $\mu\in\sP^x$, chosen by the controller in the initial state $x\in
 E$, serves as the probability law of the system's future
 evolution. Closedness under concatenation - a proxy for the
 Chapman-Kolmogorov relations - models the 
 controller's freedom to switch to a different control, within the
 admissible class, at any point $\tau\in \sT$. The set
 $\sT$ of random times  typically forms either the family of all
 deterministic times, or the family of all stopping times,
 corresponding, respectively, to Markov and strong Markov families.
 Other choices, such as the set of optional times, stopping times
 satisfying certain integrability constraints, etc., could also be of 
 interest. 

\medskip

\paragraph{Markov control problems and their value functions} A
triplet $(\sP,\sT,G)$, consisting of a controlled Markov family
$(\sP,\sT)$ and a random variable $G\in\slzerone(\sP)=\slzerone(\cup_x
\sP^x)$ is called a \define{Markov control problem}. Its
\define{value function} $v:E\to [-\infty,\infty]$ is defined by
\begin{equation}
  \label{equ:value}
  \begin{split}
	v(x) = \textstyle\inf_{\mu\in\spx} \EE^{\mu}[G],
	\text{ for }x\in E.
 \end{split}
\end{equation}
Given $\eps>0$, a probability measure $\mu\in\spx$ is said to be
\define{$\eps$-optimal} if $\EE^{\mu}[G] < v(x) + \eps$, with the
usual convention that $-\infty+\eps=-1/\eps$. 

\medskip

A Markov control problem is simply a controlled
  Markov family together with a minimally integrable random variable
  $G$, which we interpret as the cost associated with a state of the
  world. The controller's objective - quantified by the value function
  $v$ - is to minimize the expected value of this cost.  

\medskip

\paragraph{Tail random variables and (approximate) disintegrability}\ We
give here a general definition of a class of random variables which posess
a mild shift-invariance property.

\smallskip

\begin{definition}
  \label{def:tail}
  Given a  controlled Markov family $(\sP,\sT)$, a
  Borel-measurable random variable
  $G\in \slzerone(\sP)$ is said to be a $(\sT,\sP)$-\define{tail
  random variable} if, for all $(x,\mu)\in\gp$, $\nu\in\selp$ and
  $\tau\in \sT$ we have
  \begin{equation}
  \label{equ:cct-G-tail}
	\begin{split} \EE^{\mu\cct \nu}[ G] = \EE^{\mu}[
	  g(X_{\tau})\inds{\tau<\infty}+G \inds{\tau=\infty}], \text{
	  where $g(x) = \EE^{\nu_x}[G]$. }
	\end{split}
  \end{equation}
\end{definition}

\smallskip

Tail random variables can be loosely interpreted as
  those that depend only on the future after each $\tau\in\sT$. They
  will be used as cost functionals for our main control problem,
  as defined below.  One could argue that only tail random variables
  matter, as far as Markov control problems are concerned, since there
  is little point in applying different controls to a system whose
  final cost is already known.  It will be shown in Proposition
  \ref{pro:sufficient} below that, under mild additional assumptions,
  a random variable $G$ with $G\circ \theta_t = G$, for all $t$, is a
  tail random variable. 

\medskip

\begin{definition} \label{def:cmf} 
  Given a controlled Markov family $(\sP,\sT)$, we say that a random
  variable $G$ is \define{$(\sP,\sT)$-approximately disintegrable} if
  $G$ is a $(\sP,\sT)$-tail random variable and 
  \begin{equation}
	\label{equ:disint}
	\begin{split}
	  \forall\, (x,\mu)\in
  \gp, \, \tau\in\sT, \, \eps>0,\quad \exists\, 
	\nu\in \sS(\sP),\quad 
     \EE^{\mu\cct\nu}[ G ] \leq \EE^{\mu}[ G ] +\eps.
   \end{split}
  \end{equation}
  $G$ is called $(\sP,\sT)$-\define{disintegrable} if
  \eqref{equ:disint} holds for $\eps=0$, as well. 
\end{definition}

\medskip

  The notion of (approximate) disintegrability is, in a sense, dual to
  that of closedness under concatenation of Definition
  \ref{def:conMarfam}. If a family, closed under concatenation,
  \emph{allows} the controller to change his/her mind at any point,
  (approximate) disintegrability states that the controller can remain
  (near)  optimal if \emph{forced} to do so.  We also note, for future
  use, that approximate disintegrability allows us to construct
  $\eps$-optimizers of the form $\mu\cct \nu$, for $\mu\in\sP^x$,
  $\nu\in\sS(\sP)$ and $\tau\in \sT$.

\subsection{An abstract dynamic-programming principle}
With the notions of a Markov control problem, (approximate)
disintegrability and the value function introduce above, we are now ready to
state our abstract version of the DPP. The reader familiar with the
discrete-time theory will observe that, once the framework has been
set up, the proof is quite standard and, for the most part, follows
the well-known steps.

\smallskip

\begin{theorem}[An Abstract dynamic programming Principle]  
  \label{thm:main}
  Let $(\sP,\sT,G)$ be a Markov control problem with $G$  
  $(\sP,\sT)$-approximately disintegrable, and let $v$ be its value
  function.
  Then:
  \begin{enumerate}
	\item The  function $v$ is lower semi-analytic and, hence,
	  universally measurable.
	\item For each $\eps>0$ there exists a universally-measurable
	  selector $\hne\in \selp$ such that $\hne_x$ is $\eps$-optimal
	  for each $x\in E$.
	\item The dynamic-programming principle (DPP)  holds at all
	  $\tau\in\sT$; more precisely, following the convention that
	that $\EE^{\mu}[ Y ]=\infty$ as soon as $\EE^{\mu}[Y^+]=\infty$, we
        have
	  \begin{equation}
	  \label{equ:DPP}
	  \begin{split} v(x) = \inf_{\mu\in\spx} 
		\EE^{\mu}\Big[ v(
		  X_{\tau})\inds{\tau<\infty} + G\inds{\tau=\infty}
		\Big]\text{ for
		all }x\in E\eand \tau\in\sT,
	  \end{split}
	  \end{equation}
\end{enumerate}
\end{theorem}

\medskip

\noindent{\em Proof.} 

1.
	  The statement clearly holds if $v(x)=+\infty$, for all $x\in
	  E$.  Otherwise, for $a\in (\inf_{x\in E} v(x),\infty)$, we
	  construct the sub-level set
	  \[ \sets{x\in E}{ v(x)< a} = \sets{x\in E}{ \EE^{\mu}[G] < a
		\text{ for some } \mu\in\spx}=\proj_E J^{<a},
	  \] 
  where $\proj_E$ denotes the natural projection from 
  $E\times \prob$ onto $E$, and \[ J^{<a}=\Bsets{(x,\mu)}{
	\EE^{\mu}[G^+]< a+\EE^{\mu}[G^-]} 
  \bigcap \Bsets{(x,\mu)}{ \mu\in\spx} \subseteq E\times \prob.\]
  Since the maps $\mu\mapsto \EE^{\mu}[G^+]$ and $\mu\mapsto \EE^{\mu}[G^-]$
  are Borel measurable and 
  graph $\gp=\sets{(x,\mu)}{ \mu\in\spx}$ is analytic, both sets in the
  definition of $J^{<a}$ are analytic. Therefore, so are  
  the set $J^{<a}$ and its projection $\set{v<a}$.
  Consequently, $v$ is lower semi-analytic and, a fortiori, universally
  measurable.

\smallskip

  2.
  Let the set 
  $\gpe$ be given by 
  \[ \gpe=\Bsets{ (x,\mu)\in E\times \prob}{ \EE^{\mu}[G^+]
  \leq \EE^{\mu}[G^-]+(v(x)+ \eps) }\bigcap \gp,\]
  where the convention $-\infty+\eps = -1/\eps$ is
  used for $v(x)+\eps$. 
  Since $v$ is universally measurable, $\gpe$ 
  is $\sU\times \sB(\prob)$-measurable. 
Thanks to the fact that the
universal $\sigma$-algebra is closed under the Suslin operation
(\cite[Theorem 3.5.22., p.~114]{Sri98}), we can
apply (a slight generalization of) the Jankov-von Neumann theorem 
(\cite[Theorem 5.7.5, p.~200]{Sri98}) and conclude that a
``universally-measurable section'', i.e., a kernel
$\hne$ with $\Gamma_{\hne} \subseteq \Gamma_{\sP^{\eps}}$ exists. 
Therefore, $\hne\in\sS(\sP)$, and, clearly,
 $\hne_x$ is $\eps$-optimal for each $x\in E$. 

\smallskip

3. To simplify the notation, we use the convention 
  $g(X_{\tau})=v(X_{\tau})=G$ on $\set{\tau=\infty}$. 
  Given  $x\in E$, $\tau\in \sT$ and $\eps>0$, the assumption of 
  approximate disintegrability allows us to  
  choose an $\eps$-optimal measure  of the form $\hmu \cct \hnu$,  
  for some $\hmu\in \spx$ and $\hnu\in\sS(\sP)$. The universally-measurable
map $g:E\to (-\infty, \infty]$, defined by 
$g(x) = \EE^{\hnu_x}[G]$ for $x\in E$ satisfies $g(x) \geq v(x)$, for
all $x\in E$.  On the other hand, the tail property of $G$, as in
\eqref{equ:cct-G-tail}, yields 
  \begin{equation}
  \nonumber
  \begin{split}
	v(x)+\eps > \EE^{\hmu\cct \hnu}[G]  &= 
	\EE^{\hmu}[g(X_{\tau})].
   \end{split}
  \end{equation}
  Without loss of generality we assume that $v(x)<\infty$, so that 
  for large enough $n\in\N$,
  \[ v(x) + \eps > \EE^{\hmu}[ \max( -n, g(X_{\tau}))] \geq
  \EE^{\hmu}[ \max( -n, v(X_{\tau})].\] Consequently,
  with the convention
  $v(X_{\tau})=G$ on $\set{\tau=\infty}$, we have
  \[ v(x) \geq \inf_{\mu\in\spx} \inf_{n\in\N} \EE^{\mu}[
  \max(-n,v(X_{\tau}))],\text{ for all } x\in E. \]
  
  \medskip

  For the opposite inequality, 
  using (2) above, for each $\eps>0$ we can 
  pick a kernel  $\hne\in \sS(\sP)$ 
  such that $\hne_x$ is $\eps$-optimal for each $x\in E$,  
  and set $g(x) = \EE^{\hne_x}[G]$ so that $g$ 
  is universally measurable and $g (x) \leq v(x) + \eps$, 
  for all $x\in E$. 
  For $x_0\in E$ with $v(x_0) > -\infty$, 
  an arbitrary $\mu\in\sP^{x_0}$, and $n\in\N$, we have
\begin{equation}
\nonumber
\begin{split}
  v(x_0) & \leq \EE^{\mu\cct\hne}[ G] 
  = \EE^{\mu}[ g(X_{\tau})] 
  \leq  \EE^{\mu}[ \max(-n, g(X_{\tau}))]\\ & \leq
  \EE^{\mu}[ \max(-n, v(X_{\tau})+\eps)].
 \end{split}
\end{equation}
Assuming, without loss of generality, that
$\EE^{\mu}[ v(X_{\tau})^+]<\infty$, we use the
Dominated convergence theorem to conclude that
\begin{equation}
  \label{equ:ineq4}
  \begin{split}
	v(x_0) \leq 
	\Beeq{\mu}{ \max\big(-n, v(X_{\tau})\big)}.
 \end{split}
\end{equation}
To complete the proof, we take the infimum over all
$\mu\in\sP^{x_0}$ and all $n\in\N$. \qquad\endproof

\medskip

  The analyticity assumption of Definition \ref{def:conMarfam} of a
  controlled Markov family plays a major role in the proof of the
  parts (1) and (2) of Theorem \ref{thm:main}. In contrast, closedness
  under concatenation and the (approximate) disintegrability are only
  used in part (3), one for each of the two opposite inequalities
  which consitute \eqref{equ:DPP}.

\subsection{A sufficient condition for disintegrability}\ 
Next, we present a simple sufficient condition for
disintegrability and the tail property, applicable in
most cases of interest.  For  $x\in E$, a random time $\tau$ and a
probability measure $\mu$ on $\de$, 
we let 
\[ (\xi,B)\mapsto \mu[ \theta_{\tau}\in B| X_{\tau} = \xi], \quad (\xi,B)\in
E\times \bde, \] 
denote a regular version of the 
conditional distribution of $\theta_{\tau}$, given $X_{\tau}$, under
$\mu$; its 
 existence - as a map defined 
$\mu\circ X_{\tau}^{-1}$-a.s.~- 
 is guaranteed by the Polish property of $\de$
(see, e.g., Theorem 5.3, p.~84.,~in \cite{Kal02b}). We remind the
reader that $X_{\tau}=\iota$ on $\tau=\infty$, and that, under any
$\mu$, by convention, $\iota$ is absorbing, i.e.,
that 
$\mu[\cdot|X_{\tau}=\iota]$ is the Dirac mass on $\omega_{\iota}$ when
$\mu[X_{\tau}=\iota]>0$.

\medskip

\begin{proposition}[A sufficient condition for disintegrability]
  \label{pro:sufficient}
  Let $(\sP,\sT)$ be a controlled Markov family, and let $G$ be a
  Borel random variable in $\slzerone(\sP)$ such that
  \begin{equation} \label{equ:(a)}
	  G(\theta_t(\omega)) = G(\omega)\text{ for all } t\geq 0,\,
	  \omega\in\de.
  \end{equation}
  Then $G$ is a $(\sP,\sT)$-tail random variable. 
  
  If, in addition,  
  for each $\mu\in \cup_x \sP^x$ and  $\tau\in\sT$ we have 
	  \begin{equation}
		\label{equ:(b)}
		\begin{split}
	  	\mu[ \theta_{\tau}\in \cdot|X_{\tau}=x]\in
		\sP^{x},\text{ for
	  $\mu\circ X_{\tau}^{-1}$-almost all $x \in E$,  }
	   \end{split}
	  \end{equation}
	  then
  $G$ is $(\sP,\sT)$-disintegrable. 
\end{proposition}

\medskip

\noindent\emph{Proof.}
First, we prove that a random variable $G\in\slzerone(\sP)$, which
satisfies \eqref{equ:(a)},  is a tail random variable.  
The relations in  \eqref{equ:compose} imply that  for all
$\omega,\omega'\in\de$, and any random time $\tau$, we have
$G(\omega\cct \omega') = G(\omega')$ \new{as soon as $\tau(\omega)<\infty$ and
$X_{\tau}(\omega) = X_0(\omega')$.} 
Therefore, for $(x,\mu)\in\gp$, $\tau\in\sT$, $\nu\in\sS(\sP)$, and
$\omega\in \de$ such that $\tau(\omega)<\infty$, integration against
$\nu_{X_{\tau}(\omega)}$ yields 
\[ 
  \int G(\omega\cct \omega')\,
	\nu_{X_{\tau}(\omega)}(d\omega') = g(X_{\tau}(\omega)),\text{
  where } g(x) = \int G(\omega')\, \nu_{x}(d\omega').
\] 
When $\tau(\omega)=\infty$, we have $\int G(\omega\cct \omega')\,
\nu_{X_{\tau}(\omega)}(d\omega') = G(\omega)$, and so 
\begin{equation} \nonumber \begin{split}
	  \EE^{\mu\cct\nu}[G]  =\int \int G(\omega\cct \omega')\,
	  \nu_{X_{\tau}(\omega)}(d\omega') \mu(d\omega)= \EE^{\mu}[
	  g(X_{\tau}) \inds{\tau<\infty} + G \inds{\tau=\infty}].
	\end{split} 
\end{equation}

Next, we focus on disintegrability, and fix $\mu\in\cup_x
\spx$.  \new{The condition \eqref{equ:(b)} requires that the regular version of 
$\mu[\theta_{\tau}\in \cdot |X_{\tau}=x]$ belongs to $\sP^x$, but only for
$\mu\circ X_{\tau}^{-1}$-almost all $x$. 
We can, however, easily redefine it
on a Borel subset $F$ of $\de$,
while keeping measurability in the first variable. Moreover, thanks to the 
 nonemptyness of $\sS(\sP)$, we can arrange the redefinition so that
  the newly obtained version is a kernel in $\sS(\sP)$. We 
  denote this new version by $\mu_x^{\tau}$,  with the usual interpretation that
$\mu_x^{\tau}[B]= \mu[\theta_{\tau}\in B|X_{\tau}=x]$.}

The definition of the regular
conditional distribution implies that
\begin{equation}
  \label{equ:cons-mu}
  \begin{split}
	\EE^{\mu}[ G\circ \theta_{\tau}] = 
	\iint G(\omega')
	\mu^{\tau}_{X_{\tau}(\omega)} (d\omega')\, \mu(d\omega),
  \end{split}
\end{equation}
so that, using the fact (implied by \eqref{equ:(a)})
that $G(\omega')\inds{\tau(\omega)<\infty}=
G(\omega\cct\omega')\inds{\tau(\omega)<\infty}$, for all
$\omega,\omega'\in\de$ \new{with $X_{\tau}(\omega)=X_0(\omega')$ and
$\tau(\omega)<\infty$,} we have
\begin{equation}\nonumber
 \begin{split}
   \EE^{\mu}[ G\circ \theta_{\tau}] &+
   \iint  G(\omega\cct \omega')
   \inds{\tau(\omega)=\infty} \mu^{\tau}_{X_{\tau}(\omega)} (d\omega')\,
   \mu(d\omega)\\ & = 
   \iint \Big( G(\omega\cct \omega')+
   G(\omega')\inds{\tau(\omega)=\infty}  
   \Big) \mu^{\tau}_{X_{\tau}(\omega)} (d\omega')\, \mu(d\omega)
   = \EE^{\mu\cct \nu}[G]
 \end{split}
\end{equation}
where the convention that $G(\omega_{\iota})=0$ is used. 
On the other hand, that same convention, and the assumption
\eqref{equ:(a)} yield that 
$G\circ \theta_{\tau}= G\inds{\tau<\infty}$ and 
\[
   \iint  G(\omega\cct \omega')
   \inds{\tau(\omega)=\infty} \mu^{\tau}_{X_{\tau}(\omega)} (d\omega')\,
   \mu(d\omega) = \EE^{\mu}[ G\inds{\tau=\infty}].\qquad \endproof
 \]

 \medskip

  A typical case in which \eqref{equ:(a)} in Proposition
  \ref{pro:sufficient} holds is when $G$ is of the form $G(\omega) =
  \lim_{t\to\infty} \gamma(X_{t}(\omega))$, for some Borel function
  $\gamma:E\to\R$, where the limit above should be interpreted \new{in an
  appropriate sense (see Lemma 
\ref{lem:2DE2} for a precise description).} 
  In most applications, it
  exists as an a.s.-limit with respect to all
  $\mu\in \cup_x \spx$. Almost all stochastic control problems of
  interest, such as those with the finite time horizon, running cost,
  or, more generally, cost upon absorption in finite or infinite time
  can easily be seen to fall under its domain. 

  As for \eqref{equ:(b)}, it simply states that for each control $\mu$
  and each stopping time $\tau\in \sT$, the controller can switch to a
  currently available control is such a way that that system continues
  to evolve as if no switch has been made at all.  In particular, if
  forced to make a switch, the controller can do it in such a way that
  the expected cost stays exactly the same. 

\subsection{First examples}\ 
\label{sub:fex}
To familiarize the reader with the content and scope of Definition
\ref{def:cmf}, we present several general examples in various
degrees of detail; in most cases, the full treatment is outside the
scope of this paper. 
 
 \medskip

  \paragraph{(Strong) Markov families}
In a prototypical example of a controlled Markov family all the sets
$\sP^x$ are singletons.  Indeed, for a (strong) Markov family
$(\PP^x)_{x\in E}$ on the canonical space $\de$, one can easily show
that the assignment $\sP^x=\set{\PP^x}$ for $x\in E$ defines a
controlled Markov family in the sense of Definition \ref{def:cmf}.
This can be achieved, for example, if $(\PP^x)_{x\in E}$ has the
Feller property, but the (strong) Markov property and the RCLL paths
will suffice. In that case, $\spx$, together with the set $\sT$ of
all optional times for the canonical filtration 
 -  or just deterministic times if the Markov property is not strong - 
and the appropriate
tail random variable $G$ form a controlled Markov family. While the
concatenation and disintegration properties are simple reformulations of
the Chapman-Kolmogorov equations and the (strong) Markov property (via
Proposition \ref{pro:sufficient}), the analyticity of the graph $\gp$
follows from Borel-measurability of the map $x\mapsto \PP^x$. Indeed,
the Markov property implies that (see \cite[Proposition
1.2, p.~158]{EthKur86}) the map $x\mapsto \PP^x[B]$ is
Borel-measurable for each $B\in \bde$.  
A classical result of
Varadarajan (see \cite[Lemma 2.3., p.~194]{Var63}) states that this
weak form of measurability (sometimes referred to as {customary
Borel measurability}) implies that the
map $x\mapsto \PP^x$ (with the co-domain $\prob$) - and, therefore,
its graph - is Borel measurable, and, a fortiori, analytic.

The conclusions of our main result (Theorem \ref{thm:main})
  are not novel or particularly illuminating
in this case. The lower semianalyticity in (1) can be strengthened
to Borel measurability and the existence of a universally measurable
selector in (2) is trivially verified. The third conclusion simply
reformulates the (strong) Markov property, stating that
the function of the form $v(x)=\EE^x[ G]$ has the mean-value
property. The true significance of this example is to aid intuition
by
drawing parallels between the notion of a controlled Markov
family and its stencil,  a Markov family of probability measures on
$\bde$.

\medskip

\paragraph{Multiple solutions to the martingale problem}
One of the most versatile and studied methods of constructing Feller
families on $\bde$ is through the martingale problem of Stroock and
Varadhan. Under appropriate regularity conditions (which we skip here
and refer the reader to \cite{EthKur86} or \cite{StrVar06} for
details), 
one picks a (typically local or nonlocal differential) operator $A$ defined on a class $\sA$ of functions
$f:E\to \R$ and constructs a family of families
of measures $\spxe$ on $\de$ such that
\begin{enumerate}
  \item $\mu( X_0=x)=1$, for all $(x,\mu)\in\gp$, 
  \item under $\mu$, the process
	$f(X_t) - \int_0^t Af(X_u)\, du$
	is an $\FFF^X$-martingale for each $f\in\sA$. 
\end{enumerate}
When a unique such measure $\mu=\PP^x$ exists for each $x\in E$, it
can be shown under right regularity circumstances that the family
$(\PP^x)_{x\in E}$ is a Feller family and that the
infinitesimal generator of the associated semigroup is $A$. 

When multiple solutions of the martingale problem can be found, it is
well-known that the Markov property does not necessarily hold under
all of them. \new{Very often, however,} the totality of all solutions still forms a controlled
Markov family in the sense of Definition \ref{def:cmf}. The
concatenation and disintegration properties follow from the
definition.
The analyticity - Borel measurability, in fact - 
is obtained \new{in many cases} by a
simple observation that $\spx$ is defined via a countable number of
Borel measurable restrictions \new{(see, e.g., Theorem 4.2.1, p.~86 in
\cite{StrVar06}, for the continuous case).} 

\medskip

\paragraph{Optimal stochastic control and stopping}
In the strong formulation, a typical stochastic control problem involves a set $\sN$  of adapted
(progressive, predictable, etc.) processes defined on a filtered
probability space, together with a mechanism that transforms a given
control $\nu$ into a stochastic process $X^{\nu}$ (typically defined on the same
filtered space) and a criterion by which the performance of $X^{\nu}$
is measured.  We think of $\nu$ as chosen by the controller in the effort
to affect the dynamics of the state process so as to optimize an
objective function, which is taken to be of the Meyer form $\EE[
\vp(X^{\nu}_T)]$ on a finite horizon $[0,T]$, for the purposes 
of this discussion. It is immediately clear that the fundamental
ingredient is not the process $\prf{X^{\nu}_T}$ itself, but its
distribution, which we assume can be lifted to the space $\de$, where
$E$ denotes the space in which $X^{\nu}$ takes values. In other words,
we view a control problem as an optimization problem over a set of
probability measures on $\de$, parametrized by control processes in
$\sN$.

To get a better handle on the original problem - hopefully via a dynamic
programming principle - one often embeds it in a family of 
similar problems,
parametrized by the elements of the state space (in this case
$[0,T]\times E$). 
To use the abstract Theorem \ref{thm:main}, one needs
the problems corresponding to various $(t,x)\in
[0,T]\times E$ to fit together in a way similar to
transition densities of a Markov process  
in a minimally measurable way.
Concatenation and disintegration are typically inherited from the
structure of $\sN$ - the controls can usually be concatenated pathwise as
functions, and one can
usually condition away the irrelevant past values of a control $\nu\in\sN$
to prove disintegrability. As
for the analyticity of the graph $\gp$, one needs to impose a bit of
structure on $\sN$, as well as on the distribution map, i.e., the
function which translates $\nu\in\sN$
into the distribution of $X^{\nu}$. A nice property of analyticity
is that it is preserved under Borel maps; this fact allows us to  deal
with the case where $\sN$ is a Borel space (a Borel subset of a Polish
space) and the distribution map is merely Borel. This, in particular,
covers the ubiquitous case where $\sN$ is a Borel subset of a separable Banach
space and the distribution map is defined as the solution of a
controlled SDE via the martingale-problem formalism under appropriate
regularity conditions. 

\pagebreak
\section{Applications to Financial Mathematics}

The goal of this section is to transfer the conclusions of the
abstract DPP of Theorem \ref{thm:main} to two fundamental problems of
financial mathematics.

\subsection{The financial market model}
\label{sse:finmar}
We start with a description of the underlying
financial model common to both problems. 

\medskip

\paragraph{The state space and the Markov family of ``physical''
measures}
We adopt the setting of section \ref{sec:model} with the state space
$E$ of the form
\[ E = [0,\infty) \times \R^d \times F,\]
  where $F$ is a locally compact Hausdorff space with a countable base
  (and, in particular,  Polish).
For notational reasons, we split the
components of the coordinate process $X$ as follows
\[  X_t(\omega)=(T_t(\omega),S_t(\omega),\eta_t(\omega)),\]
where $T$, $S=(S^1, S^2, \dots, S^d)$  
and $\eta$ take values in $[0,\infty)$, $\R^d$ and $F$, respectively.
The ``physical'' dynamics of $X$ will be described
via a family $(\PP^x)_{x\in E}$ of probability measures in $\prob$,
i.e., on $\de$, in the sense that, given the initial condition $x$,
$\PP^x$ models the evolution of the process $X$.

  Let $\FFF^0=\prfi{\sF^0_t}$ denote the natural (raw) filtration generated by the coordinate process $X$, and let
  $\FFF=\prfi{\sF_t}$ be its right-continuous hull, i.e., 
$\sF_t=\cap_{s>t} \sF^0_s$, for $t\geq 0$. 
We postulate Markovian dynamics by
enforcing the following assumption, where $\EE^x$ denotes the
expectation under $\PP^x$ and $\sT$ the set of all bounded
$\FFF$-stopping times:
\begin{assumption}[Markovian dynamics]
  \label{ass:Markov}
 The map $E\ni x \mapsto \PP^x\in\prob$ is Borel
 measurable, $\PP^x[X_0=x]=1$, for all $x\in E$,  and the strong Markov 
 property 
\begin{equation}
  \label{equ:Marko}
  \begin{split}
	\EE^{x}[ Z \circ \theta_{\tau}|\sF_{\tau}] & =
	\EE^{X_{\tau}}[Z], 
	\text{ $\PP^x$-a.s.,}
 \end{split}
\end{equation}
holds for all $\tau\in\sT$, $x\in E$ and all bounded Borel random
variables $Z$.
\end{assumption}

\medskip

A family $(\PP^x)_{x\in E}$ which satisfies Assumption \ref{ass:Markov}
above will be called a \define{canonical strong Markov family}. 
Such a family
automatically obeys \define{Blumenthal's 0-1 law}, i.e., that, for all
$t\geq 0$, we have
\begin{equation}
  \label{equ:Blumenthal}
  \begin{split}
	\forall\, A\in\sF_t,\ \exists\, A^0\in \sF^0_t,\ \PP^x[ A
	  \bigtriangleup
	A^0]=0.
 \end{split}
\end{equation}

\medskip

One of the most important examples of a eanonical strong Markov family
arises as the family of laws of a RCLL version of a Feller process
on a Hausdorff locally-compact topological space $E$ 
with a countable
base (LCCB). We refer the reader to \cite{RevYor99}, Chapter 3, \S
2 and \S 3,  for details. 
The assumption of the Hausdorff property is made so that the state space
$E$ is Polish -  a property that will be needed for other aspects of
the theory. 

Many models used in finance (when viewed toghether with their
factors) have the Feller property. Even more will form canonical
strong Markov families, as defined above. A large class of examples
are formed by the weak solutions of SDEs. More precisely, 
continuous solutions of uniquely-solvable local-martingale
problems (in the sense of Stroock-Varadhan) with measurable
coefficients define canonical strong Markov families (see
Proposition 18.11, p.~344., in \cite{Kal02b}). 

Many processes with jumps will also fit our framework. For example,
L\' evy processes and various transformations thereof are Feller
processes, as are solutions to a large class of SDEs driven by them
(see Section 6., Chapter V, of \cite{Pro04} for details). In fact, it
seems that all examples of Markov processes which do not give rise to
an RCLL canonical strong Markov family, while abundant and important
mathematically, feature serious pathologies as far as financial
modelling is concerned. 

Finally, even though this is not done here explicitly, there is
nothing preventing ``killed'' processes from being included in the
setup by the usual  addition of a ``cemetery'' state to $E$.

\bigskip

In addition to the canonical strong Markov 
property, three mild further assumptions will be imposed on
$(\PP^x)_{x\in E}$. The first two correspond to the usual form most
financial models take, while the third one is of technical nature
and is satisfied in most models used in practice. 

\medskip

\paragraph{Finite Horizon} The first additional assumption simply
encodes the standard trick which allows us to treat a
finite-horizon, inhomogeneous Markov process in the infinite-horizon
and homogeneous framework.  We remind the reader that a set
$E'\subseteq E$ is said to be $(\PP^x)_{x\in E}$-absorbing if for
each $x\in E'$, $\PP^x$ is the Dirac mass on the trajectory with
constant value $x$. 

\medskip

\begin{assumption}[Finite horizon]\label{ass:fin-hor}\  
 \begin{enumerate}
  \item $T_t = T_0 - t$, for $t\leq T_0$, $\PP^x$-a.s., 
	for all $x\in E$,
	and
  \item the set $E' = \set{0} \times \R^d \times F\subseteq
	E$ is $(\PP^x)_{x\in E}$-absorbing.
 \end{enumerate}
\end{assumption}

\medskip

\paragraph{Absence of arbitrage} Unlike the $T$-component of $X$, which,
thanks to Assumption \ref{ass:fin-hor}, can be interpreted as
\define{time-to-go}, the $S$-components model \define{risky-asset
prices}. The $F$-valued process $\eta$ plays the role of  an
\define{external factor} whose values drive the dynamics of $S$, but
are not necessarily themselves tradeable in a financial market
(stochastic volatility or macroeconomic indicators are two examples
out of many).  The fact that $S$ is assumed to be actively traded
leads quite naturally to the economic assumption of \emph{absence of
arbitrage}, which comes in several variants in the literature, with
the notion of No Free Lunch with Vanishing Risk (NFLVR) being probably
the dominant one. We impose it here in an equivalent form by
asking for the existence of an equivalent local-martingale measure.
More precisely, for $x\in E$, let $\sM^x$ denote the set of all
probability measures $\QQ\in\prob$, such that

\begin{enumerate}
  \item $\QQ$ and $\PP^x$ are equivalent, and
  \item the process $\prfi{S_t}$ is a $(\QQ,\FFF^x)$-local
	martingale.
\end{enumerate}

\medskip

\begin{assumption}[NFLVR] \label{ass:NFLVR} 
For each $x\in E$, $\sM^x\ne\emptyset$.
\end{assumption}

\medskip

Thanks to \cite[Theorem 5.3, p.~241]{DelSch98}, $\sigma$-marginales locally
bounded from below are local martingales. Therefore, by \cite[Theorem 1.1,
p.~215]{DelSch98},  Assumption \ref{ass:NFLVR} is equivalent to the
assumption of No Free Lunch with Vanishing Risk (NFVLR) in the financial
model $\prfi{S_t}$, on $(\Omega,\sF,\FFF, \PP^x)$, for all $x\in E$,
whenever $S$ is locally bounded from below, $\PP^x$-a.s., for all $x\in E$.
As will be shown in the sequel, the local boundedness from below will be
made a standing assumption.
          
          \smallskip

Even though the main results of arbitrage theory were written under the
usual conditions of right-continuity and completeness, we use the raw
filtration $\FFF^0$ in the definition of $\sM^x$. There is no loss of
generality as far as the (true) martingale condition is concerned since, to
switch to (the completion of) $\FFF$, it suffices to observe that $S$ is
already $\FFF^0$-adapted and use Blumenthal's 0-1 law, i.e.,
\eqref{equ:Blumenthal}. As for the localization, it is taken care of by
our third additional assumption of local boundedness from below,
introduced next. 

\medskip

\paragraph{Uniform local boundedness from below} For technical reasons which will soon
become apparent, we need to impose another assumption on our canonical
Markov family $(\PP^x)_{x\in E}$. With $S$ being the ``middle''
component of the coordinate process $X$ on $\de$, we set
\begin{assumption}[Uniform local boundedness from below]
  \label{ass:uni-loc} For each $n\in\N$, there exists a constant
  $a_n\in\R$ such that
   \begin{equation}
   \label{equ:reduce}
   \begin{split}
        \PP^x[ S^{\tau_n}\geq a_n] = 1, \text{ for all } x\in E, \ewhere \tau_n=
       \inf\sets{t\geq 0}{ \abs{S_t}\geq  n},
   \end{split}
   \end{equation}
with the convention that all inequalities involving vector-valued processes are
to be interpreted coordinatewise.
\end{assumption}

\medskip

When applied to the process $S$ - 
  interpreted as the asset-price process in a financial
  market -  Assumption \ref{ass:uni-loc} of 
  uniform local boundedness from below 
  is very mild.  Indeed, it covers most asset-price models used
  in practice, as they are mostly 
  continuous or bounded from below (thanks to 
  limited liability), and often both. Alternatively, 
  a process with jumps bounded from below also satisfies Assumption
  \ref{ass:uni-loc}.

\medskip

The sequence $\seq{\tau}$ of \eqref{equ:reduce} is a prototypical
example of a reducing sequence of stopping times for the local-martingale property, but 
such sequences are not enough in general. Indeed, it may happen that the
``reason'' for local martingality of a process is hidden in the full
filtration, but, perhaps, not in its natural filtration  (see, 
\cite[p.~57]{Str77} for an explicit example;  for a glimpse of the general
situation see the remainder of \cite{Str77}, as well as
\cite{FolPro11}). Under Assumption \ref{ass:uni-loc}, however,
the situation reverts back to its naive form:
\begin{lemma}\label{lem:Gam-mes-1} 
  Let $\QQ\in\prob$ be a probability measure equivalent to $\PP^x$,
  for some $x\in E$. Under Assumption \ref{ass:uni-loc}, $S$ is a $(\QQ,\FFF^0)$-local martingale if and only
  if $S^{\tau_n}$ is a $(\QQ,\FFF^0)$-martingale for each $n\in\N$.
\end{lemma}
\begin{proof} The proof follows a well-known argument going back at least to
  \cite{Str77}. Since it is short, we reformulate it here for convenience. 
As $\tau_n$ are stopping times and $\tau_n(\omega)\to\infty$, for all
$\omega$, we only need to prove $S^{\tau_n}$ if a $\QQ$-martingale
as soon as $S$ it is a $\QQ$-local martingale. Furthermore, we can
assume, without loss of generality, that $S$ is real-valued. 
Thanks to the inequality $(S^{\tau_n})^* \leq n+ \abs{S_{\tau_n}}
\inds{\tau_n<\infty}$, 
it is enough to show that
$S_{\tau_n}\inds{\tau_n<\infty}\in\lone(\QQ)$,
for each $n\in\N$. 
This follows, however, from the optional sampling theorem, since
$S^{\tau_n}$ is a $\QQ$-bounded-from-below $\QQ$-local martingale, and,
therefore, a $\QQ$-bounded-from-below $\QQ$-supermartingale.
\end{proof}

 \subsection{Local-martingale measures form a controlled Markov family}\ 
 Having described the model of the financial market, our next task is to
 show that the family $\sM=(\sM_x)_{x\in E}$ forms a controlled Markov
 family and that it is $G$-disintegrable for a large class of
 random variables $G$. 
 Proofs of these claims, split into several auxiliary statements, will take
 the rest of this section. 
 
 Let us reiterate that Assumptions \ref{ass:fin-hor}, \ref{ass:NFLVR} and \ref{ass:uni-loc}
 are in force throughout the section. Moreover, in the light of
 Assumption \ref{ass:fin-hor}, we can (and do) assume without loss of
 generality that all
 random times in $\sT$ are bounded. Indeed, for any 
 initial state, the system will get absorbed in finite deterministic time.

\medskip

\paragraph{Analyticity of the graph of $\sM$}
Our first task is to establish analyticity of the graph
$\gm=\sets{(x,\QQ)}{\QQ\in\sM^x}$. 
Here, as in 
the rest of this section, $\sQ$ denotes the set of all probability
measures $\QQ\in\prob$ with $\QQ\sim\PP^x$ for some $x\in E$, and the
sequence $\seq{\tau}$ is as in Definition \ref{equ:reduce}. 
  For $x\in E$ and $n\in\N$, $\sM^{x}_n$ is the set of all
  $\QQ\sim\PP^x$ under which  (each coordinate of) $S^{\tau_n}$ is a
  $\QQ$-martingale. Thanks to Lemma \ref{lem:Gam-mes-1}, we have the
  following equality
  \begin{equation}
  \label{equ:locallemma}
  \begin{split}
  	\gm = \cap_n \Gamma_{\sM_n},\ewhere \Gamma_{\sM_n}=\sets{(x,\QQ)}{
  \QQ\in\sM^{x}_n}.
   \end{split}
  \end{equation}
This way, the question of analyticity of the graph $\gm$ is reduced to 
the that of the sequence $\Gamma_{\sM^n}$, defined via the (true)
martingale property. It is in this step that the uniformity of local
boundedness in Assumption \ref{ass:uni-loc} is important.
Consequently, focusing on the
(true) martingale property, we provide a simple characterization of
the martingale property via a countable number of Borel operations;
here $Q_+$ denotes the set of all rational numbers in $[0,\infty)$.  

\smallskip

  \begin{lemma} 
	\label{lem:Gam-mes-2}
There exists a countable family $\set{A_q^n}_{q\in Q_+}^{n\in\N}$
such that $A_q^n\in \sF^0_{q}$ for all $q\in Q_+$, $n\in\N$, with 
the following property: given $\QQ\in\sQ$, a bounded-from-below and
$\FFF^0$-adapted RCLL process $Y$ is an $\bF$-martingale under $\QQ$ if and only if 
\begin{equation} \label{equ:cond-count}
	\begin{split}
  	\EE^{\QQ}[ Y_r \ind{A_q^n}] = \EE^{\QQ}[ Y_q \ind{A_q^n}],\text{ for all $q\leq r\in Q_+$ and $n\in\N$. }
   \end{split}
  \end{equation}
\end{lemma}
\begin{proof}
  Necessity of \eqref{equ:cond-count} is clear, so we focus on
  sufficiency. 
  The $\sigma$-algebra $\sF^0_q$ is (bimeasurably) isomorphic to the Borel
$\sigma$-algebra generated by the topology of the Polish space
$\de[0,q]$ of $E$-valued RCLL paths on $[0,q]$. Hence, it is countably
generated, and we can choose an enumeration $\fam{A_q^n}{n\in\N}$ of a
countable generating set.  By adding all finite intersections if
necessary, we can assume further, without loss of generality,  that
this generating set is a $\pi$-system, and that it
contains $\de$.  

Since $Y_0$ is $\QQ$-a.s.-constant for each $\QQ\in \sQ$, the
condition \eqref{equ:cond-count} implies that, for $r\in Q_+$, we have $Y_r\in\lone(\QQ)$ and
\begin{equation}
  \label{equ:rat-mart}
  \begin{split}
	\EE^{\QQ}[ Y_r|\sF^0_q]=Y_q, \text{ $\QQ$-a.s., for all $q\leq r\in
Q_+$}.
 \end{split}
\end{equation}
By Blumenthal's law \eqref{equ:Blumenthal}
$\sF^0_q$ and $\sF_q$ differ in $\QQ$-trivial sets only. Thus, the  martingale
property in \eqref{equ:rat-mart} holds even if we replace $\sF^0_q$ by
$\sF_q$. Also, for a fixed $r\in Q_+$, the 
family $\fam{Y_q}{ q\in Q_+,q\leq r}$ is
$\QQ$-uniformly integrable. Hence, by the right continuity of
$Y$, so is any family of the form $\fam{Y_t}{t\leq r}$, for $r\geq 0$.
An approximation from the right yields that \[ \EE[
Y_t|\sF_q]=Y_q,\text{ $\QQ$-a.s., for each } t\geq 0, q\in Q_+, q\leq
t.\] Another right approximation  - this time in $q$ - and the
backward martingale convergence theorem imply that \[ \EE[
Y_t|\FF_{s}]=Y_s,\text{ $\QQ$-a.s., for all } 0\leq s \leq t <
\infty.\] Therefore, $\prfi{Y_t}$ is an $\bF$-martingale and, a
fortiori, an $\bF^0$-martingale.  
\end{proof}

\medskip

After the local martingale property, we turn to the question of
Borel-measurability of the measure-equivalence relation $\sim$ 
on $\prob$. 
\new{The following auxiliary result appears to be well-known, but a precise
reference has been hard to locate. We give a proof for completeness:}
\begin{lemma}
  \label{lem:Gam-mes-3}
Let $C$ be a fixed countable dense set in $\de$ 
and let $\sB$ be the (countable) family of all sets of the form 
$\cup_{i=1}^m B(x_i,1/n)$  for some finite family $x_1,\dots, x_m$ in
$C$ and some $n\in\N$. 
  For $\PP,\QQ\in\prob$, we have $\QQ\nll\PP$ if and only if
  \begin{equation} \label{equ:nll}
	\begin{split}
  	\exists\, N\in\N\quad \forall\, n\in\N\quad \exists\, B_n \in
	\sB\quad \PP[B_n]\leq \oo{n} \eand\, \QQ[B_n]\geq \oo{N}.
   \end{split}
  \end{equation}
\end{lemma}
\begin{proof}
  If \eqref{equ:nll} holds it is standard to show that
  $\QQ\nll\PP$.  Conversely, to show that $\QQ\nll\PP$ implies
  \eqref{equ:nll}, we use the
  fact that all finite measures on Borel sets of a Polish space are regular \cite[Theorem 12.7, p.~438]{AliBor99} to
  pick a compact set $K$ and a constant $N\in\N$, such that
  $\PP[K]=0$ and $\QQ[K]\geq 1/N$.  For $n\in\N$, let
  $\sB_n$ denote 
  the set of all open balls with centers in $C$ and radii $1/n$, and
  let the family $\seq{\sB^K}$ of families of open balls be defined by 
  \[ \sB_n^K= \sets{ B\in \sB_n}{ B\cap K\ne\emptyset},\efor n\in\N.\]
  By compactness, for each $n\in\N$ we can find a finite sub-collection $B^n_1,\dots, B^n_{m_n}$ in $\sB^n_K$ such that $K\subseteq B_n=\cup_{k=1}^{m_n} B^{n}_k$. Since $B_n\subseteq \sets{x\in X}{ d(x,K)\leq 1/n}$, we have $\PP[B_n]\to 0$, as $n\to\infty$. On the other hand $\QQ[B_n]\geq \QQ[K]\geq 1/N$, for all $n\in\N$. 
\end{proof}

\medskip

\new{The following statement can be derived as a direct consequence of
\cite[Theorem 58, p.~52]{DelMey82}. We include a short self-contained proof for
completeness.}
\begin{corollary}
  \label{cor:simBorel}
  The graph $\Gamma_{\sim}=\sets{(\PP,\QQ)\in\prob^2}{ \PP\sim\QQ }$ of the
  measure-equivalence relation $\sim$ is a 
  Borel subset of $\prob^2$. 
\end{corollary}
\begin{proof}
It suffices to notice that
Lemma \ref{lem:Gam-mes-3} states that we can express the graph 
  $\Gamma_{\sim}=\sets{(\PP,\QQ)}{ \QQ\sim\PP }$ of the relation
  $\sim$ using only Borel operations starting from the sets of the
  form 
  $\sets{(\PP,\QQ)\in\prob^2}{ \PP[A]\leq \alpha}$ and
  $\sets{(\PP,\QQ)\in\prob^2}{ \QQ[A]\geq \beta}$. That these subsets of
  $\prob^2$ are Borel measurable follows from a combination of the
  Portmanteau theorem and a monotone class argument.
\end{proof}

\medskip

The reader should compare
our next proposition to a related, independent, dis\-cre\-te-time result (namely Lemma 4.8) in
\cite{BouNut13}.
\pagebreak
\begin{proposition}\label{pro:analytic-super}
  The graph $\gm$ is an analytic subset of $E\times \prob$.
\end{proposition}
\begin{proof}
  By \eqref{equ:locallemma}, we need to prove that
  $\Gamma_{\sM_n}$ is analytic, for each $n\in\N$. We fix $n\in\N$ and
  observe that
   $\gmn = \Gamma_1\cap\Gamma_2$, where 
\begin{equation}\nonumber
 \begin{split}
   \Gamma_1=E\times \Bsets{\QQ\in \prob}{ S^{\tau_n}\text{ is a
   $\QQ$-martingale}} \eand \Gamma_2=\Bsets{(x,\QQ)\in E\times \prob}{\QQ \sim
   \PP^x}.
 \end{split}
\end{equation}
By Lemma \ref{lem:Gam-mes-2},  there exists a Borel set $\sR\subseteq \prob$
such that 
\[ \Bsets{\QQ\in \prob}{ S^{\tau_n}\text{ is a $\QQ$-martingale}} \cap \sQ
= \sR \cap \sQ.\]
Since $\Gamma_2\subseteq E\times \sQ$, it will be enough to show that
$\Gamma_2$ is analytic. 

  It is a part of the definition of a canonical $\sT$-Markov family
  that the map 
  $x\mapsto \PP^x\in\prob$ is Borel. Therefore, so are its graph
  $\Gamma_{\PP}=\sets{(x,\PP^x)}{ x\in E}\subseteq E\times\prob$ and 
  the product $\Lambda_1 = \Gamma_{\PP} \times \prob$.  
  By Corollary \ref{cor:simBorel},  $\Lambda_2 = E \times
  \Gamma_{\sim}$ is Borel, and, hence, so is 
  \[ \Lambda= \Lambda_1\cap \Lambda_2 = \sets{ (x,\PP^x,\QQ)}{
  \QQ\sim\PP^x}.\]
  It remains to observe that $\Gamma_2$ is the (canonical) projection
  of $\Lambda$ onto the first and third coordinates and conclude that
  it is an analytic set in $E\times \prob$. 
\end{proof}

\medskip

\paragraph{Closedness under concatenation and disintegrability}
We remind the reader that $\sS(\sM)$ denotes the set of all
(universally measurable) kernels, and that each $\nu\in\sS(\sM)$ can
be interpreted as a universally-measurable map $x \mapsto \cup_{x\in
E} \sM^x$ with $\nu_x\in\sM^x$ for all $x\in E$.  
\begin{proposition}
  \label{pro:concat}
For all $(x,\QQ)\in\gm$, $\tau\in \sT$ and $\nu\in\sS(\sP)$, we have
$$\QQ\cct\nu\in\sM^x.$$
\end{proposition}

 \emph{Proof.} Thanks to Assumption \ref{ass:uni-loc} 
 and by stopping at some $\tau_n$ (as in \eqref{equ:reduce}), we can 
 assume that $S$ is bounded from below and a (true) martingale
 under each $\QQ\in\cup_x \sM^x$. Also, by considering each
 component separately, we may assume that $S$ is one-dimensional,
 i.e., that $d=1$.

We fix $(x,\QQ), \tau$ and $\nu$ as in the statement, and note that,
by direct computation, 
$\qcn\sim\PP^x$. 
Then, we pick a $0\leq s\leq t$ and a bounded random variable $F\in\sF_{\tau+s}$ and observe
that, for all $\omega\in\de$, the random variable $\omega'\mapsto
F(\omega\cct\omega')$ is $\sF_s$-measurable. Therefore, by
the $\bF$-martingale property of $S$ under $\nu_x$, we have 
 \[ \int S_{t}(\omega') F(\omega\cct\omega')\, \nu_x(d\new{\omega'})=
\int S_s(\omega') F(\omega\cct\new{\omega'})\, \nu_x(d\new{\omega'}),\]
for all $\omega\in\de$ and $x\in E$.
The identity $S_r(\omega')=
S_{\tau+r}(\omega\cct\omega')$, 
valid for $\nu_{X_{\tau}(\omega)}$-almost all $\omega'$ 
used with $r=s$ and $r=t$ 
implies that $\EE^{\qcn}[ S_{\tau+t}
F]=\EE^{\qcn}[ S_{\tau+s} F]$. So, 
\begin{equation}
  \nonumber 
  \begin{split}
	\EE^{\qcn}[ S_{\tau+t}|\FF_{\tau+s}] = S_{\tau+s},
	\text{ $\qcn$-a.s., for all $0\leq s \leq t$.}
 \end{split}
\end{equation}
In words, $S$ is a $\qcn$-martingale ``after $\tau$''. On the other hand,
$\QQ$ and $\qcn$ coincide on $\sF_{\tau}$, which implies immediately
that $S$ is a $\qcn$-martingale ``before $\tau$''. Thanks to the
optional sampling theorem, the two properties combine, and we conclude
that $S$ is a $\QQ\cct\nu$-martingale. Indeed, for a bounded
stopping time $\kappa$ we have the following chain of inequalities, where
all the expectations are under $\qcn$:
\begin{equation}
\nonumber
\begin{split}
  \EE[ S_{\kappa}] 
  &= \EE[S_{\kappa\wedge\tau}\inds{\kappa\leq \tau} + 
      S_{\kappa\vee \tau} \inds{\kappa>\tau}]
	  = \EE[ \EE[S_{\tau}|\FF_{\kappa\wedge\tau}] \inds{\kappa\leq\tau}
	  +\EE[ S_{\kappa\vee \tau} \inds{\kappa>\tau}|\FF_{\tau}]]\\
	  &= \EE[ S_{\tau}\inds{\kappa\leq\tau}] +
	  \EE[S_{\tau} \inds{\kappa>\tau}] = \EE[S_{\tau}].\qquad\endproof
 \end{split}
\end{equation}

\begin{proposition}
  \label{pro:disint}
  The family $\smxe$ satisfies the condition \eqref{equ:(b)} of
  Proposition \ref{pro:sufficient}.
\end{proposition}
\begin{proof}
  We pick $(x,\QQ)\in\gm$, $\tau\in\sT$, and recall that $\tau$ is bounded. As in the proof of
  Proposition \ref{pro:concat}, we assume that 
  $S$ is a one-dimensional $\QQ$-martingale bounded from below,
  for each $\QQ\in\cup_{x\in E} \sM^x$.
 
  Let $q_x$ be the $\QQ$-regular conditional
  distribution of $\theta_{\tau}$, given $X_{\tau}=x$. 
  By constructing the Radon-Nikodym density, for example, we immediately
  note that $q_x\sim \PP^x$, for $\QQ\circ X_{\tau}^{-1}$-almost all $x\in E$. 
  Given a stopping time $\sigma\in\sT$, we set
  $\kappa=\tau+\sigma\circ\theta_{\tau}$. Since both $\tau$ and
  $\sigma$ are stopping times for the shifted raw filtration
  $\prfi{\sF^0_{t+\eps}}$, for each $\eps>0$, Galmarino's test (see,
  e.g., Exercise 4.21, p.~47 in \cite{RevYor99}), implies that
  $\kappa$ has the same property, and so,  $\kappa\in \sT$. 
  Therefore, $\EE^{\QQ}[S_{\kappa}|\sigma(X_{\tau})] = S_{\tau}$, $\QQ$-a.s. Moreover, 
$S_{\kappa} =
S_{\sigma}\circ\theta_{\tau}$, and so, 
\[  \EE^{\QQ}[ f(X_{\tau}) S_{\tau}]= 
  \EE^{\QQ}[ f(X_{\tau})S_{\kappa}] = \EE^{\QQ}[ 
f(X_{\tau}) \textstyle \int S_{\sigma}(\omega')\, q_{X_{\tau}}(d\omega')],\]
for all bounded Borel functions $f:E\to\R$. In particular, it follows
that
\[ \int S_0(\omega') q_{X_{\tau}(\omega)}(d\omega') = S_{\tau}(\omega) = \int
  S_{\sigma}(\omega') q_{X_{\tau}(\omega)}(d\omega'),\text{
for $\QQ$-almost all $\omega\in\de$.}\]
Therefore $S_{\sigma}$ and $S_0$ have the same expectation under
$q_x$, for $\QQ\circ X_{\tau}^{-1}$-almost all $x$, with the
exceptional set possibly depending on the stopping time $\sigma$.
By Lemma \ref{lem:Gam-mes-2}, however, it suffices to consider a
countable collection of stopping times of the form $\sigma =
q\ind{A^n_q}+r \ind{(A^n_q)^c}$, for rational $q<r$ and $A^n_q$ as in
the statement of Lemma \ref{lem:Gam-mes-2} 
we conclude that $S$ is 
$q_x$-martingale for $\QQ\circ X_{\tau}^{-1}$-almost all $x\in E$. 
\end{proof}

\subsection{Lower hedging}\ 
\label{sse:lower-hedging}
We adopt the framework and notation of subsection \ref{sse:finmar}
and turn to the problem of lower hedging (sub-hedging) in a financial
market. \new{First, we define a random variable $G$  - with a pointwise 
shift-invariance property - which will play a role
of the contingent claim to be sub-replicated,  The main ingredient is the
following simple observation which states, loosely, that our state space
admits a ``limit superior'' (a weaker, but measurable, notion similar to
that of a Banach
limit).
\begin{lemma}
\label{lem:2DE2} 
There exists a measurable functional $L:\de\to E$ such that
\begin{enumerate}
\item $L(\omega) = \lim_{t\to\infty} X_t(\omega)$, whenever the limit
exists, and
\item $L( \theta_t(\omega)) = L(\omega)$, for all $\omega\in \de$ and all
$t\geq 0$. 
\end{enumerate}
\end{lemma}

\smallskip

\begin{proof}
$E$ is a metrizable noncompact LCCB  space, so its Alexandroff (one-point)
compactification is Polish (see 
\cite[Theorem 3.44, p.~92]{AliBor99}). Therefore, it can be embedded into
the Hilbert cube $H=[0,1]^{\N}$ as a compact set. In fact, we identify $E$ with its
copy in $H$, and we denote the image of the ``infinity
point'' by $\infty$. Let $p_n$, $n\in\N$, denote the coordinate projections on $H$.
For $\omega\in \de$ 
we define $L'(\omega)\in H$ by requiring that
\[ p_n(L'(\omega)) = \limsup_{t\to\infty} p_n(X_t(\omega)),\text{ for all
}n\in\N,\]
so that, immediately, we have $L'(\theta_t(\omega))=L(\omega)$, for all
$t\geq 0$ and $\omega\in\de$. 
The continuity of each $p_n$ and the RCLL property of $X_t$ guarantee the Borel
measurability of the map $\omega\mapsto \limsup_{t\to\infty}
p_n(X_t(\omega))$, for each $n\in\N$, and, therefore, of the map $L'$.
Thanks to the compactness of $E\cup \set{\infty}$, a RCLL path $t\mapsto
X_t$ in $E$ will converge to $x\in E\cup\set{\infty}$ if and only if each of
the $[0,1]$-valued paths $p_n(X_t)$ converges to $p_n(x)$ as $t\to\infty$.
It follows that $L'(\omega) = \lim_{t\to\infty} X_t(\omega)$, whenever this
limit exists and that $L'(\omega)\in E\cup \set{\infty}$ otherwise. It
remains to remap $\infty$ to an arbitrary point in $E$, i.e., to define the
map $L$ by 
$L=\kappa \circ L'$, where $\kappa$ is a measurable ``retraction''
$\kappa:E\cup\set{\infty}\to E$.
\end{proof}}

\smallskip

\new{
  With $L$ as in Lemma \ref{lem:2DE2} above and a Borel
function $\vp:E\to \R$, we set
\[ G(\omega) = \tilde{\vp}(L(\omega)),\]
where $\tilde{\vp}(t,S,\eta) = \vp(S,\eta)$, for all $(t,S,\eta)\in E$. 
Since under each $\PP^x$, the coordinate process
$X=(T,S,\eta)$ gets absorbed
after $T_0$ units of time, and the absorption point is 
$L(\omega)$, $\PP^x$-a.s., it is clear that
$G$ can be interpreted as the value of a derivative claim written on the risky
asset $S$ as well as the factor $\eta$, with maturity $T_0$.
With the standing interpretation of the first coordinate process
$T$ as ``time to go'', $G$ should be understood $\PP^x$-a.s.~as
a composition of a Borel
function (the payoff function) with the values of the risky asset $S$
and the factor $\eta$ at maturity $T_0$ units of time from now.
Moreover, we have $G\circ \theta_t = G$, for all $t$; indeed, this property
is directly inherited from $L$. }

Under the minimal assumption that $\EE^{\QQ}[ G^-]<\infty$, for all
$\QQ\in\cup_x \sM^x$, for a given 
$x=(T,S,\eta)\in [0,\infty)\times \R^d\times F$, we define
  the \define{lower-hedging price} of $G$ as
\begin{equation}
  \label{equ:superhedging}
  \begin{split}
	v(T,S,\eta) = \inf_{\QQ\in\sM^{T,S,\eta}} \EE^{\QQ}[G].
 \end{split}
\end{equation}

\medskip

  One usually (and more naturally) defines the lower hedging 
  price $v(x)=v(t,S,\eta)$ as
  follows:
  \[ v(T,S,\eta)= \sup \Bsets{ y\in\R}{ \exists H\in \AA^{T,S,\eta},\
  y+\int_0^{T} H_u\, dS_u \geq \vp(S_T, \eta_T),\ \PP^{T,S,\eta}\text{ -  a.s}},\]
  where $\AA^{T,S,\eta}$ denotes the set of all admissible strategies. The two
formulations are equivalent (see \cite[Theorem 5.12, p.~246]{DelSch98}), 
but we
prefer the one in \eqref{equ:superhedging} as it fits our framework
much better.

\smallskip

\begin{theorem}[DPP for lower hedging] 
  \label{thm:lower-hedging}
  Under Assumptions \ref{ass:Markov}, \ref{ass:fin-hor},
  \ref{ass:NFLVR} and \ref{ass:uni-loc}, and with $v$ defined in
  \eqref{equ:superhedging}, we have 
  \begin{enumerate}
	\item The value function $v:[0,\infty)\times \R^d \times F\to
  [-\infty,\infty]$ is universally measurable and the
  dynamic-programming equation (DPP) holds:
  \[ v(T,S,\eta) = \inf_{\QQ\in\sM^{T,S,\eta}} \EE^{\QQ}[ v(T-\tau,
S_{\tau}, \eta_{\tau})],\quad v(0,S,\eta) = \vp(\eta,S),\]
for all $\FFF$-stopping times $\tau\leq T$, where we define $\EE^{\QQ}[Y]=\infty$, as
soon as $\EE^{\QQ}[Y^+]=\infty$. 
\item For $\eps>0$, an
$\eps$-optimal $\hat{\QQ}^{T,S,\eta}\in\sM^{T,S,\eta}$ 
can be associated to each $(T,S,\eta)\in [0,\infty)\times \R^d \times E$ in a
universally-measurable way. 
\end{enumerate}
\end{theorem}

\smallskip

\begin{proof}
  According to Propositions \ref{pro:analytic-super} and
  \ref{pro:concat}, the pair $(\sM,\sT)$ forms a controlled Markov
  family, with $\sT$ denoting the set of all bounded $\FFF$-stopping
  times. Moreover, since $G\circ\theta_t=G$, the statement of
  Proposition \ref{pro:disint} guarantees that the conditions of
  Proposition \ref{pro:sufficient} are satisfied; hence, $G$ is
  $(\sT,\sM)$-disintegrable. That, in turn,
  allows for our main abstract result, Theorem \ref{thm:main}, to be
  applied. 
\end{proof}

\subsection{Utility maximization}\ 
The utility-maximization problem - its generalized dual formulation, to be more
precise - can be phrased in the framework almost identical to that of
subsection \eqref{sse:lower-hedging}. Indeed, given $z\geq 0$ we
 pose the \define{generalized dual utility-maximization problem:}
\begin{equation}
  \label{equ:util-max-prob}
  \begin{split}
	v(T,S,\eta,z) &
	= \inf_{\QQ\in\sM^{T,S,\eta}}
	 \EE^{T,S,\eta}[ V(S_T,\eta_T,z \tRN{\QQ}{\PP^{T,\eta,S}})],
  \end{split}
\end{equation}
where $\EE^{T,S,\eta}$ denotes the expectation under
$\PP^{T,S,\eta}$, and $V:[0,\infty)^d\times F \times [0,\infty) \to\R$
	is a Borel function with the property that
	 $\EE^{T,S,\eta}[ V^-(S_T,\eta_T,z
	 \tRN{\QQ}{\PP^{T,\eta,S}})]<\infty$, for all
	 $\QQ\in\sM^{T,S,\eta}$.

Unfortunately, the objective function in \ref{equ:util-max-prob} is not
of the form $\EE^{\QQ}[G]$ for a random variable $G$ defined on $\de$, so
our main abstract result cannot be directly applied.  It will be possible
to do that, however, if we include $[0,\infty)$ as an additional factor
in the definition of the state space: 
\[ 
  \tilde{E} =  E\times [0,\infty)=[0,\infty) \times [0,\infty)^d \times
		F\times [0,\infty).
\]
We identify the space $\dep$ with the product space $\de\times \dz$,
where $E = [0,\infty)\times [0,\infty)^d \times F$ is the state space of
the lower-hedging problem.  To simplify the notation,  the coordinate
process on $E$ will be denoted by $X$ and that on $\tE$ by $\tX=(X,Z)$; a generic point on $E$ is $x=(T,S,\eta)$, while $\tx=(x,z)$ is a generic point in
$\tE$.  For a fixed element $x\in E$ and a martingale measure
$\QQ\in\sM^{x}$,  using $\de$ as the underlying probability space,
we construct a RCLL version of the Radon-Nikodym density $Z^{\QQ}$ of $\QQ$
with respect to $\PP^{x}$: 
\[ 
  Z^{\QQ}_t = \EE^{x}[\tRN{\QQ}{\PP^{x}}|\sF_t], \text{ for }t\geq 0.
\]

For $x\geq 0$, the  $\PP^{x}$-law of the of vector of processes
$(X,zZ^{\QQ})=(T,S,\eta,z Z^{\QQ})$ on $\de$ defines a Borel probability
measure on $\dep$, which we denote by $\PP^{T,S,\eta,z;Q}$, or, simply, by
$\PP^{x,z;\QQ}$.  It is characterized by the following equality \[
\EE^{\PP^{x,z;\QQ}}[ G( \Xd, \Zd) ] = \EE^{x}[ G( \Xd,
zZ^{\QQ}_{\cdot})],\] valid for all bounded Borel $G:\dep\to\R$.
For $(x,z)\in E'$, we define 
\begin{equation} \label{equ:def-sP}
\begin{split} 
  \tsP^{x,z} = \bsets{\PP^{x,z;\QQ}}{\QQ\in\sM^{x}}\subseteq \tprob,
\end{split} \end{equation} 
where $\tprob$ denotes the set of all probability measures on $\dep$.
The natural filtration
generated by $\tX$ on $\dep$ is denoted by $\tFFF^0$, and its
right-continuous hull by $\tFFF$; their raw and right-continuous
sub-filtrations, generated by $X$, will be denoted by $\tFFF^{X;0}$
and $\tFFF^{X}$, respectively. As usual, $Q_+$ is the set of
nonnegative rational numbers. 

\smallskip

  \begin{proposition}
	\label{pro:con-Z}
	The family $\tsP^{x,z}$, $(x,z)\in\tE$ is closed under concatenation.
  \end{proposition}

  \smallskip

\begin{proof}
We pick a probability $\PP^{x_0,z_0;\QQ_0}$, a stopping time
$\tilde{\tau}$ and a kernel $(\hPP^{x,z:\hQQ^{x,z}})_{x,z}$ into
elements of $\tsP$. It is immediate that that $(\hQQ^{x,z})_{x,z}$ is
also a kernel, from $\tilde{E}$ into probability measures on $\dep$.
To avoid multiple levels of indexation, we shorten
$\PP^{x_0,z_0;\QQ_0}$ to $\PP^0$ and denote by $\EE^0$ the
corresponding expectation operator. Also, the process $Z^{\QQ^0}$ will
be denoted simply by $Z^0$. 
	
For $q\in Q_+$, 
we define $\tilde{B}_q = \set{ \tilde{\tau}< q} \in
\tilde{\sF}^0_q.$ Since $Z^{0}$ is $\tFFF^X$-adapted, there exists an
$\tilde{\sF}^X_q$-measurable set $\tB'_q$ such that $\PP^{0}[ \tB'_q \bigtriangleup
\tilde{B}_q]=0$. Each set $\tB'_q$, $q\in Q_+$ can be further
identified with an event $B_q\in \sF_q$ on $\de$ such that, 
with the 
  $\FFF$-stopping time $\tau$ on $\de$ defined by 
\[ \tau(\omega) = \sup\sets{q\in Q_+}{\omega\in B_q},\]
we have, for each bounded Borel $H:\dep\to\R$, 
  \[ \EE^{0}[ H(X_{\tilde{\tau}\wedge \cdot },
	Z_{\tilde{\tau}\wedge\cdot})] = \EE^{x} [ 
  H(X_{\tau\wedge\cdot}, z Z^{\QQ}_{\tau\wedge\cdot})],\]
  where the expectation on the left is over $\dep$ and the one on the
  right over $\de$. 
	
We would like to prove that $\PP^0 \cctt \hPP^{\cdot} = \PP^{x_0,z_0;
\hat{\QQ}}$, for some $\hQQ\in\MM^{x_0}$, with
$\QQ|_{\sF_{\tau}}=\QQ_0|_{\sF_{\tau}}$.  Consider the family $\sA$ of
all random variables of the product form $\lht$
where $L=L(\Xd,\Zd),H=H(\Xd,\Zd)$ are bounded and Borel on $\dep$, and
$L$ is $\sF^{0}_{\ttau}$-measurable.  By the monotone-class theorem, it
will be enough to show that the two measures act on $\sA$ in the same
way. For $L\, H\circ\theta_{\ttau}\in \sA$  we have
\[ \EE^{\PP^0\cctt \hPP^{\cdot}}[ \lht ]=
	\EE^0[ L  h(X_{\ttau}, z_0Z^{0}_{\ttau})],\]
where $h(x,z) = \EE^{\hPP^{x,z}}[ H] = \EE^{x}[ H( \Xd, z
\Zd^{\QQ^{x,z}})]$, i.e.,
\[ \EE^{\PP^0\cctt \hPP^{\cdot}}[ \lht]=
  \int L(\omega)
  \EE^{X_{\ttau}(\omega)}[H(\Xd(\omega'), 
  z_0 Z_{\tau}^{0}(\omega) Z^{\QQ^{X_{\tau}(\omega),Z_{\tau}^{0}(\omega)}}_{\cdot}(\omega'))]
  \,\PP[d\omega].\]

If, for $t\geq 0$, we  set $\hZ_t= Z_{t\wedge \tau}^{0}, Z_{t-t\wedge
\tau}^{\QQ^{X_{\tau}, Z_{\tau}^{0}}}$, it is clear that $\hZ$ is an
$\FFF$-adapted, RCLL martingale. Moreover, it has the easy-to-verify
property that (each component of) $\hZ S^{\tau_n}$ is a martingale on
$\de$ for 
each $n\in\N$, with $\seq{\tau}$ given in \eqref{equ:reduce}.
Finally, since $\hZ_{\infty}>0$,
$\PP^x$-a.s., there exists a probability measure $\hQ\in\sM^x$ such
that $\hZ =Z^{\hQ}$.  The reader will now readily check that
$\EE^{\PP^0\cctt \hPP^{\cdot}}[ \lht] = 
\EE^{x_0,z_0;\hQ}[ \lht]$,
for all $\lht\in \sA$, and, consequently, that
$\PP^0\cctt \hPP^{\cdot} = \PP^{x_0,z_0;\hQ}\in \tsP^{x_0,z_0}$.
\end{proof} 

\medskip

Using the corresponding statement for the family $\sM^x$, namely
Proposition \ref{pro:disint} and the ideas from the above proof of
Proposition \ref{pro:con-Z}, we get the following result:

\smallskip

\begin{proposition} \label{pro:disint-Z}
  The family $\tsP^{x,z}$, $(x,z)\in E'$ is
  $G$-disintegrable, for all $G$ satisfying \eqref{equ:(a)}.
\end{proposition}

\medskip

\new{Finally, we turn to the proof of analyticity of the graph $\tgp$ of 
$(\sP^{x,z})_{x,z}$. The reader will note that some
parts of the construction used in the proof of Proposition \ref{pro:anal-z} below are 
quite similar to the central argument in the proof of existence of a
jointly measurable Radon-Nikodym derivative in Theorem 58, p.~52 of 
\cite{DelMey82}. }

\smallskip

\begin{proposition}
  \label{pro:anal-z}
  The graph $\tgp=\sets{(x,z,\PP)}{(x,z)\in \tE, \
  \PP\in\tsP^{x,z}}\subseteq \tE \times \tprob$ is analytic.
\end{proposition}

\smallskip

\begin{proof}
  Let $\hat{\prob}$ denote the set of all pairs $(\PP,\QQ)$ of probability
  measures on $\de$ with $\QQ\ll\PP$. For each such pair, let
  $\zeta(\PP,\QQ)$ denote the $\PP$-joint law of the pair $(X,\zpq)$, 
  on $\dep$, where $\zpq$ is the RCLL version of the martingale
  \[ \zpq_t = \EE^{\PP}[ \tRN{\QQ}{\PP} | \sF_t].\]
Using the stability under Borel maps of the analytic sets, it is 
enough to show that the map $\zeta$ is Borel. As already commented on
in subsection \ref{sub:fex}, part (2), by the classical result  
\cite[Lemma 2.3, p.~194]{Var63} of Varadarajan, it suffices to
  show that $\zeta$ is customarily Borel measurable, i.e., that the map
  \[ (\PP,\QQ) \mapsto \zeta(\PP,\QQ)[B],\] is Borel, for each Borel $B$
  on $\dep$.
  The $\pi$-$\ld$ theorem can be used to reduce this
  further to product cylinders. More precisely, it is enough to
  show that the map
  \begin{equation}
	\label{equ:todp}
	\begin{split}
	  (\PP,\QQ)&\mapsto \PP[ (X_{t_1},\zpq_{t_1})\in C_1\times D_1,\dots,
	  (X_{t_n},\zpq_{t_n})\in C_n\times D_n]
   \end{split}
  \end{equation}
  is Borel, for each $n\in\N$, and all $C_1,\dots, C_n\in \sB(E)$ and
  $D_1,\dots, D_n\in \sB(\R)$. It turns out to be more convenient,
  but equally valid, to consider bounded continuous functions $f:
  \tE^n\to\R$ and the maps
  \begin{equation}
	\label{equ:todp1}
	\begin{split}
	  (\PP,\QQ)&\mapsto \eeq{\PP}{ f\Big( (X_{t_1},Z^{\PP,\QQ}_{t_1}),
  \dots (X_{t_n},Z^{\PP,\QQ}_{t_n})\Big)}.
   \end{split}
  \end{equation}
 By approximation, it suffices 
 to show that maps of the form
  \begin{equation}
	\label{equ:todp2}
	\begin{split}
	  (\PP,\QQ)&\mapsto \eeq{\PP}{ F( X_{t_1},X_{t_2}, \dots, X_{t_n})
	\rho(Z^{\PP,\QQ}_{t_1},\dots, Z^{\PP,\QQ}_{t_n})}
   \end{split}
  \end{equation}
are Borel, for all nonnegative and bounded $F$ and continuous and
convex $\rho$. We can write $\rho$ as a supremum of a countable number
of affine functions on $\R^n$, so the problem reduces to the
Borel-measurability of the map of the form
  \begin{equation}
	\label{equ:todp3}
	\begin{split} 
	  (\PP,\QQ)&\mapsto \eeq{\PP}{ F( X_{t_1},X_{t_2}, \dots, X_{t_n})
	\zpq_{t}},
   \end{split}
  \end{equation}
  for all $t\geq 0$.
  Using the fact that $\sF^0_t=D_E[0,t]$, we can construct a nested sequence
  $\seq{\sK}$ of finite partitions of $\sF^0_t$ with the property that
  $\sigma( \cup_n \sK_n ) = \sF^0_t$. 
  For each $n\in\N$ and each pair $(\PP,\QQ)\in \prob^2$, 
  we define a $\sigma(\sK_n)$-measurable random
  variable $Z^{\PP,\QQ}_n$ on $\de$ as follows:
  \[ Z^{\PP,\QQ}_n = \sum_{i=1}^{\abs{\sK_n}}
  \tfrac{\QQ[A^n_i]}{\PP[A^n_i]} \ind{A^n_i},\]
  with $\sK_n=\set{A^n_1,\dots, A^n_{\abs{\sK_n}}}$ and 
  $\tfrac{\QQ[A^n_i]}{\PP[A^n_i]}=0$, as soon as $\PP[A^n_i]=0$.
  The martingale-convergence theorem implies that $Z_n^{\PP,\QQ}\to
  \zpq_t$, $\PP$-a.s., which, in turn, establishes the
  Borel-measurability in \eqref{equ:todp3} and completes the proof.
\end{proof}

\medskip

Propositions \ref{pro:con-Z}, \ref{pro:disint-Z} and \ref{pro:anal-z}
allow us to apply the abstract DPP of Theorem \ref{thm:main}
in the present setting. 

\smallskip

\begin{theorem}[DPP for utility maximization]
  \label{thm:utility}
  Under Assumptions \ref{ass:Markov}, \ref{ass:fin-hor},
  \ref{ass:NFLVR} and \ref{ass:uni-loc}, and with $v$ and $V$ as
  in equation \eqref{equ:util-max-prob} and below it, we have
  \begin{enumerate}
	\item The value function $v:[0,\infty)\times \R^d \times F\times
		[0,\infty) \to
  [-\infty,\infty]$ is universally measurable and the
  dynamic-programming equation (DPP) holds:
  \[ v(T,S,\eta,z) = \inf_{\QQ\in\sM^{T,S,\eta}} \EE^{\PP}[ v(T-\tau,
  S_{\tau}, \eta_{\tau}, z Z^{\QQ}_{\tau})],\quad v(0,S,\eta,z) = V(\eta,S,z),\]
  for all $\FFF$-stopping times $\tau\leq T$, where $Z^{\QQ}_{\tau} =
  \eecq{\PP^{T,S,\eta}}{ \tRN{\QQ}{\PP^{T,S,\eta}}}{\sF_{\tau}}$ and 
with the convention that $\EE^{\QQ}[Y]=\infty$, as
soon as $\EE^{\QQ}[Y^+]=\infty$. 
\item For $\eps>0$, an
$\eps$-optimal $\hat{\QQ}^{T,S,\eta,z}\in\sM^{T,S,\eta}$ 
can be associated to each $(T,S,\eta,z)$ $\in [0,\infty)\times \R^d
  \times E\times [0,\infty)$ in a
universally-measurable way. 
\end{enumerate}
\end{theorem}
  
  \smallskip
  
  We focus on the dual formulation of the utility-maximization problem
  in this paper not only 
  because it fits our framework, but also because it is often much
  easier to work with in practice. For the reader interested in the
  primal problem, here are a few words about its relationship with the
  dual problem and the DPP.  
  
The classical utility-maximization theorem with random endowment in as
described in \cite{CviSchWan01} and adadpted to fit our notation comes
with the function $V$ of the form \[ V(S,\eta,z) = \tU(z) + z
\vp(S,\eta),\] where $\vp$ is a bounded Borel function and
$\tU(z)=\sup_{\xi \geq 0} \Big( U(\xi) - \xi z \Big)$ is the dual of a
utility function $U$, i.e., an increasing and strictly concave
$C^1$-map $U:(0,\infty)\to \R$, with $U'(0+)=\infty$ and
$U'(\infty)=0$. Additionally, the condition of reasonable asymptotic
elasticity $\AsE[U]<1$ is imposed, where
$\AsE[U]=\limsup_{x\to\infty} xU'(x)/U(x)<1$ when $U(x)>0$ for
large enough $x$ and $\AsE [U]=0$, otherwise. Under these
conditions, and a suitable finiteness assumption, the authors of
\cite{CviSchWan01} show 
that the following conjugate relationship holds
	  \begin{equation}
		\label{equ:conj}
		\begin{split}
	  	u(T,S,\eta,\xi) = \inf_{z\in Q_+} \Big( v(T,S,\eta,z) - \xi z \Big),
	   \end{split}
	  \end{equation}
where $u$ is the value function of the \emph{primal
utility-maximization problem} given by \[ u(T,S,\eta,\xi) = \sup_{H\in
\sA^{T,S,\eta}} U\left( \xi+ \int_0^T H_u\, dS_t + \vp(S_T,\eta_T)
\right),\] with $\sA^{T,S,\eta}$ denoting the set of admissible
portfolio processes, and where $T,S,\eta$ and $\xi$ 
range over the problem's effective
domain. The infimum in \eqref{equ:conj} is taken over a countable
set, so the conclusion of universal measurability transfers directly
from $v$ to $u$. The same relationship can be used to show that 
the function $u$ satisfies the appropriate version of the DPP.

\appendix
\section{Some notions from descriptive set theory} For technical
reasons, some basic concepts from descriptive set theory are needed.  The
following few paragraphs provide a (necessarily inadequate) review of the
most prominent notions and facts; we refer the reader to Chapter 7 of
\cite{BerShr78}, or the textbooks \cite{Sri98} and \cite{Kec95}, for a
thorough treatment.

A \define{Polish} space is a separable topological space whose
topology can be induced by a complete metric. The main examples
include all countable discrete spaces, as well as Euclidean spaces
$\R^d$, $d\in\N$, in addition to a host of ``infinite-dimensional''
spaces such as the Hilbert space $[0,1]^{\N}$, the Baire space
$\N^{\N}$, or any separable Banach space, all  with the usual topologies.
Compact metric spaces as well as locally-compact Hausdorff spaces with
a countable base (LCCB) are Polish; so are countable products
(with the product topology) of Polish spaces and all their
$G_{\delta}$ (and, in particular, open or closed) subsets. 

A subset $A$ of a Polish space $X$ is called \define{analytic} if it
can be realized as a projection of a Borel subset of $X\times \R$ onto
$X$.
Borel subsets of Polish spaces are analytic,
but, unless the space is countable, one can always find non-Borel
analytic sets. The family of all analytic sets is closed under
countable unions and intersections, as well as direct and
inverse images of Borel maps. It
is, generally, not closed under complementation. In fact, a complement
of an analytic set, known as a co-analytic set,
is itself analytic if and only if it is Borel measurable.  

The \define{universal $\sigma$-algebra $\sU$} is the intersection of
the $\mu$-completions of the Borel $\sigma$-algebra on $X$, over the
family of all probability measures $\mu$ on it. Its importance for us
stems from the fact that $\sU$-measurable functions can be integrated
with respect to any probability measure, thus allowing us to treat
them as if they were Borel measurable for most practical purposes. 

A \define{Suslin scheme} on a set $X$ equipped with a family $\sG$ of
(not necessarily all of) its subsets, is a map $g: \cup_{n\in\N}
\N^{n} \to \sG$. The Suslin operation assigns to every Suslin
scheme $g$ on $\sG$ the set $A = \cup_{\sigma\in\N^{\N}} \cap_{n}
g(\sigma(1),\dots, \sigma(n))$. Two important facts about Suslin
schemes in our context are the following: 1) $A$ is an analytic subset
of a Polish space $X$ if and only if it is the result of some Suslin
operation on the closed subsets of $X$, and 2) the universal
$\sigma$-algebra is closed under the Suslin operation. It follows that
analytic sets are universally measurable.


\def\ocirc#1{\ifmmode\setbox0=\hbox{$#1$}\dimen0=\ht0 \advance\dimen0
  by1pt\rlap{\hbox to\wd0{\hss\raise\dimen0
  \hbox{\hskip.2em$\scriptscriptstyle\circ$}\hss}}#1\else {\accent"17 #1}\fi}
  \ifx \cprime \undefined \def \cprime {$\mathsurround=0pt '$}\fi\ifx \k
  \undefined \let \k = \c \fi\ifx \scr \undefined \let \scr = \cal \fi\ifx
  \soft tundefined \def \soft
  {\relax}\fi\def\ocirc#1{\ifmmode\setbox0=\hbox{$#1$}\dimen0=\ht0
  \advance\dimen0 by1pt\rlap{\hbox to\wd0{\hss\raise\dimen0
  \hbox{\hskip.2em$\scriptscriptstyle\circ$}\hss}}#1\else {\accent"17 #1}\fi}
  \ifx \cprime \undefined \def \cprime {$\mathsurround=0pt '$}\fi\ifx \k
  \undefined \let \k = \c \fi\ifx \scr \undefined \let \scr = \cal \fi\ifx
  \soft tundefined \def \soft {\relax}\fi

\end{document}